\documentclass[a4paper]{amsart}
\usepackage{amsmath,amsthm,amssymb,latexsym,epic,bbm,comment}
\usepackage{graphicx,enumerate,stmaryrd}
\usepackage[all,2cell]{xy}
\xyoption{2cell}

\usepackage[active]{srcltx}
\usepackage[parfill]{parskip}
\UseTwocells

\newtheorem{theorem}{Theorem}
\newtheorem{lemma}[theorem]{Lemma}

\newtheorem{corollary}[theorem]{Corollary}
\newtheorem{proposition}[theorem]{Proposition}

\font\sc=rsfs10
\newcommand{\cC}{\sc\mbox{C}\hspace{1.0pt}}
\newcommand{\cM}{\sc\mbox{M}\hspace{1.0pt}}

\newcommand{\cS}{\sc\mbox{S}\hspace{1.0pt}}

\newcommand{\cP}{\sc\mbox{P}\hspace{1.0pt}}
\newcommand{\cA}{\sc\mbox{A}\hspace{1.0pt}}

\font\scc=rsfs7
\newcommand{\ccC}{\scc\mbox{C}\hspace{1.0pt}}

\newcommand{\ccA}{\scc\mbox{A}\hspace{1.0pt}}

\begin{document}

\title[Morita theory for finitary $2$-categories]{Morita theory for finitary $2$-categories}
\author{Volodymyr Mazorchuk and Vanessa Miemietz}
%\date{\today}

\begin{abstract}
We develop Morita theory for finitary additive $2$-representations of finitary $2$-categories.
As an application we describe Morita equivalence classes for $2$-categories of projective
functors associated to finite dimensional algebras and for $2$-categories of Soergel bimodules.
\end{abstract}

\maketitle

\section{Introduction}\label{s0}

Classical Morita theory (see \cite{Mo}) describes when two rings have equivalent categories of representations. 
By  now there are many generalizations of this theory by varying what is represented and where it is represented
(see e.g. \cite{Ri,Ke,Kn,BV,BD,To}). 

In \cite{MM,MM2,MM3}, motivated by recent success of higher categorical methods in both topology 
(see e.g. \cite{Kh,St}) and representation theory (see e.g. \cite{Ar,Gr,LLT,CR}), we started a systematic study 
of the $2$-representation theory of finitary $2$-categories. The latter should be thought of as $2$-analogues
of finite dimensional algebras. Assuming the existence of adjunction morphisms (which should be thought of as $2$-analogues of an involution on a finite dimensional algebra) we
constructed certain natural classes of $2$-representations and established, under some natural conditions, 
a $2$-analogue of Schur's lemma as well as analogues of other criteria for simplicity of a representation.
In this article we drop the assumption on existence of adjunction morphisms and study Morita theory for arbitrary
finitary $2$-categories.

In view of the different known versions of  Morita theory (as in \cite{Ri,Ke,Kn,BV,BD,To}), our main results are
as expected. Each finitary $2$-category is biequivalent to the (opposite of the) endomorphism $2$-category of 
its representable (principal) $2$-representations. Our main result asserts that Morita equivalent 
finitary $2$-categories can be obtained one from the other by taking the (opposite of the)  endomorphism category
of a suitable ``projective generator'' or, in other words, by adding and/or removing retracts of 
principal $2$-representations. Along the way we obtain a full classification of all retracts of
principal $2$-representations and show that indecomposable retracts are given by indecomposable $1$-morphisms
which square to an idempotent (the latter might decompose). In line with the common phenomenon that 
``useful'' categorifications of semisimple algebras and categories are usually not semisimple (see e.g. \cite{CR,St}), 
retracts of projective $2$-representations do not necessarily split off as direct summands.

As an application we describe Morita equivalence classes for finitary $2$-categories of projective 
functors associated to finite dimensional algebras. This is an important class of finitary $2$-categories which
originally appeared in \cite{MM}, and which later, in \cite{MM3}, played the role of an important prototype of 
a certain class of ``simple''
finitary $2$-categories appearing in an analogue of the Artin-Wedderburn Theorem. As it turns out, the 
classification is not obvious and Morita equivalence classes are described in terms of adding/removing 
semi-simple direct summands (under some additional restrictions). We also show that for $2$-categories of 
Soergel bimodules the Morita equivalence classes correspond to isomorphism classes of 
Coxeter systems.

The paper is organized as follows: In Section~\ref{s1} we recall the classical Morita theory for
finitary $\Bbbk$-linear categories. Section~\ref{s2} is a brief introduction to the $2$-representation theory
of finitary $2$-categories. In Section~\ref{s4} we define and classify projective $2$-representations.
Although our main techniques come from category theory (see \cite{BD}), we also use the classical combinatorial
description of idempotent matrices with non-negative integer coefficients from \cite{Fl}.
In Section~\ref{s3} we prove our main result. We complete the paper with some examples and applications
in Section~\ref{s7}.

\vspace{5mm}

\noindent
{\bf Acknowledgment.} A substantial part of the paper was written during a visit of the second author to 
Uppsala University, whose  hospitality is gratefully acknowledged. The visit was supported by the Swedish 
Research Council. The first author is partially supported by the 
Swedish Research Council and the Royal Swedish Academy of Sciences. The second author is partially 
supported by ERC grant PERG07-GA-2010-268109 and by EPSRC grant EP/K011782/1. We are really grateful to the
referees for very careful reading of the paper, for pointing out several subtle inaccuracies in the
original version and for many very helpful suggestions.

\section{Morita theory for finitary $\Bbbk$-linear categories}\label{s1}

In this section we present some well-known results from classical Morita theory for finite dimensional
associative algebras in a form that resembles our subsequent treatment of Morita theory for finitary $2$-categories.
We denote by $\mathbb{N}$ the set of positive integers.

\subsection{Finitary $\Bbbk$-linear categories}\label{s1.1}

Throughout the article, let $\Bbbk$ be an algebraically closed field. A $\Bbbk$-linear  category 
is a category enriched over the category $\Bbbk\text{-}\mathrm{Mod}$ of $\Bbbk$-vector spaces, meaning that morphisms form $\Bbbk$-vector spaces and
composition of morphisms is $\Bbbk$-bilinear. A $\Bbbk$-linear category $\mathcal{C}$ is called {\em finitary} provided
that
\begin{itemize}
\item $\mathcal{C}$ is skeletally finite, that is it has finitely many isomorphism classes of objects;
\item all morphism spaces in $\mathcal{C}$ are finite dimensional (over $\Bbbk$).
\end{itemize}
A finitary $\Bbbk$-linear category $\mathcal{C}$ is called {\em reduced} if we additionally have
\begin{itemize}
\item for any $\mathtt{i}\in \mathcal{C}$ the (finite dimensional unital) $\Bbbk$-algebra $\mathcal{C}(\mathtt{i},\mathtt{i})$
is local.
\end{itemize}

\subsection{Representations of finitary $\Bbbk$-linear categories}\label{s1.2}

Let $\mathcal{C}$ be a finitary $\Bbbk$-li\-ne\-ar category. 
By a representation of $\mathcal{C}$ we mean a  $\Bbbk$-linear
functor $\mathsf{M}:\mathcal{C}\to \Bbbk\text{-}\mathrm{Mod}$. A {\em finitary} representation of $\mathcal{C}$ is a 
$\Bbbk$-linear functor $\mathsf{M}:\mathcal{C}\to \Bbbk\text{-}\mathrm{mod}$, where $\Bbbk\text{-}\mathrm{mod}$ denotes the category of finite-dimensional $\Bbbk$-vector spaces. Given  two representations 
$\mathsf{M}$ and $\mathsf{N}$ of $\mathcal{C}$, a morphism $\alpha:\mathsf{M}\to \mathsf{N}$ is just a 
natural transformation from  $\mathsf{M}$ to $\mathsf{N}$. All (finitary) representations of $\mathcal{C}$ together
with morphisms between them form an abelian category denoted $\mathcal{C}\text{-}\mathrm{Mod}$ (resp. $\mathcal{C}\text{-}\mathrm{mod}$).

One important example of a representation of $\mathcal{C}$ is given by 
the representable functor $\mathcal{C}(\mathtt{i},{}_-)$ 
for $\mathtt{i}\in \mathcal{C}$. For $\mathsf{M}\in \mathcal{C}\text{-}\mathrm{Mod}$ we have the Yoneda isomorphism
\begin{displaymath}
\mathrm{Hom}_{\mathcal{C}}\big(\mathcal{C}(\mathtt{i},{}_-),\mathsf{M}\big)\cong \mathsf{M}(\mathtt{i})
\end{displaymath}
given by $\varphi\mapsto \varphi(\mathbbm{1}_{\mathtt{i}})$. In particular, as 
$\mathsf{M}\mapsto \mathsf{M}(\mathtt{i})$ is exact, it follows that the representation
$\mathcal{C}(\mathtt{i},{}_-)$ is projective. We denote by $\mathcal{C}\text{-}\mathrm{proj}$
the full additive subcategory of $\mathcal{C}\text{-}\mathrm{mod}$ consisting of all projective 
$\mathcal{C}$-modules and note that $\mathcal{C}\text{-}\mathrm{proj}$ coincides with the full additive
closure (by which we mean closure under direct sums, direct summands, and isomorphisms) 
of the $\mathcal{C}(\mathtt{i},{}_-)$ for $\mathtt{i}\in \mathcal{C}$. Further, any indecomposable
projective $\mathcal{C}$-module is isomorphic to $\mathcal{C}(\mathtt{i},{}_-)\cdot e$, where 
$e$ is an indecomposable idempotent in $\mathcal{C}(\mathtt{i},\mathtt{i})$ (the latter being canonically isomorphic to 
$\mathrm{End}_{\mathcal{C}}(\mathcal{C}(\mathtt{i},{}_-))^{\mathrm{op}}$).

A progenerator for $\mathcal{C}\text{-}\mathrm{mod}$ is a full subcategory $\mathcal{P}$ in
$\mathcal{C}\text{-}\mathrm{proj}$ such that the additive closure of $\mathcal{P}$ coincides with
$\mathcal{C}\text{-}\mathrm{proj}$. 

\subsection{Morita theorem}\label{s1.3}

For reduced $\Bbbk$-linear categories the Morita theorem looks as follows:

\begin{proposition}\label{prop102}
Let $\mathcal{A}$ and $\mathcal{C}$ be two reduced finitary $\Bbbk$-linear categories. Then the categories 
$\mathcal{A}\text{-}\mathrm{mod}$ and $\mathcal{C}\text{-}\mathrm{mod}$ are equivalent if and only if 
the categories $\mathcal{A}$ and $\mathcal{C}$ are equivalent.
\end{proposition}

\begin{proof}
Let $\Phi:\mathcal{A}\to \mathcal{C}$ be an equivalence with inverse $\Psi:\mathcal{C}\to \mathcal{A}$. These functors give rise to 
functors
\begin{displaymath}
{}_-\circ\Phi: \mathcal{C}\text{-}\mathrm{mod}\to \mathcal{A}\text{-}\mathrm{mod}\quad\text{ and }\quad 
{}_-\circ\Psi: \mathcal{A}\text{-}\mathrm{mod}\to \mathcal{C}\text{-}\mathrm{mod} .
\end{displaymath}
These latter functors define inverse equivalences as required.

Conversely, let $\Phi: \mathcal{C}\text{-}\mathrm{mod}\to \mathcal{A}\text{-}\mathrm{mod}$
and $\Psi: \mathcal{A}\text{-}\mathrm{mod}\to \mathcal{C}\text{-}\mathrm{mod}$ be mutually inverse equivalences.
They restrict to equivalences between $\mathcal{C}\text{-}\mathrm{proj}$ and $\mathcal{A}\text{-}\mathrm{proj}$.

Let $\mathtt{i}_1,\mathtt{i}_2,\dots,\mathtt{i}_k$ be a fixed set of representatives of the isomorphism classes
of objects in $\mathcal{C}$ and let $\mathcal{C}'$ be the full subcategory of $\mathcal{C}$ consisting of
these objects. Then $\mathcal{C}'$ is a skeleton of $\mathcal{C}$ and hence is equivalent to $\mathcal{C}$.
Let $\mathcal{X}_{\mathcal{C}}$ be the full subcategory of $\mathcal{C}\text{-}\mathrm{proj}$ with objects
$\mathcal{C}(\mathtt{i}_t,{}_-)$ for $t=1,2,\dots,k$. The Yoneda isomorphism gives an isomorphism of categories
between $\mathcal{X}_{\mathcal{C}}$ and $(\mathcal{C}')^{\mathrm{op}}$. Similarly we define
$\mathcal{A}'$ and $\mathcal{X}_{\mathcal{A}}$ and get an isomorphism between 
$\mathcal{X}_{\mathcal{A}}$ and $(\mathcal{A}')^{\mathrm{op}}$. Note that $\mathcal{X}_{\mathcal{C}}$ is a 
multiplicity free additive generator of $\mathcal{C}\text{-}\mathrm{proj}$. The equivalence $\Phi$ thus maps it
to a multiplicity free  additive generator of $\mathcal{A}\text{-}\mathrm{proj}$. Hence
$\Phi(\mathcal{X}_{\mathcal{C}})$ is isomorphic to $\mathcal{X}_{\mathcal{A}}$ which implies that the categories 
$\mathcal{C}'$ and $\mathcal{A}'$ are isomorphic. Therefore $\mathcal{C}$ and $\mathcal{A}$ are equivalent.
\end{proof}

In the general case we have the following.

\begin{theorem}[Morita Theorem for $\Bbbk$-linear categories]\label{mthm}
Let $\mathcal{A}$ and $\mathcal{C}$ be two finitary $\Bbbk$-linear categories. Then the following assertions
are equivalent.
\begin{enumerate}[$($a$)$]
\item\label{mthm.1} There exists a progenerator $\mathcal{P}\in \mathcal{C}\text{-}\mathrm{mod}$
such that $\mathcal{P}^{\mathrm{op}}$ is equivalent to $\mathcal{A}$.
\item\label{mthm.2} The categories $\mathcal{A}\text{-}\mathrm{proj}$ and $\mathcal{C}\text{-}\mathrm{proj}$ 
are equivalent.
\item\label{mthm.3} The categories $\mathcal{A}\text{-}\mathrm{mod}$ and $\mathcal{C}\text{-}\mathrm{mod}$ 
are equivalent.
\end{enumerate}
\end{theorem}

\begin{proof}
Claim \eqref{mthm.3} implies claim \eqref{mthm.2} as any categorical equivalence sends projective objects to
projective objects. If $\Phi:\mathcal{A}\text{-}\mathrm{proj}\to \mathcal{C}\text{-}\mathrm{proj}$ is an
equivalence, then the image of the full subcategory in $\mathcal{A}\text{-}\mathrm{mod}$ with objects 
$\mathcal{A}(\mathtt{i},{}_-)$, $\mathtt{i}\in \mathcal{A}$, under $\Phi$ is a progenerator 
for $\mathcal{C}\text{-}\mathrm{proj}$. Hence claim \eqref{mthm.2} implies claim \eqref{mthm.1}. 
Finally, if there exists a progenerator $\mathcal{P}\in \mathcal{C}\text{-}\mathrm{mod}$
such that $\mathcal{P}^{\mathrm{op}}$ is equivalent to $\mathcal{A}$, then the functors
\begin{displaymath}
\mathcal{P}\otimes_{\mathcal{A}}{}_-: \mathcal{A}\text{-}\mathrm{mod}\to \mathcal{C}\text{-}\mathrm{mod}
\quad\text{ and }\quad
\mathrm{Hom}_{\mathcal{C}}(\mathcal{P},{}_-):\mathcal{C}\text{-}\mathrm{mod}\to \mathcal{A}\text{-}\mathrm{mod} 
\end{displaymath}
are easily seen to be mutually inverse equivalences of categories
(for the definition of $\mathcal{P}\otimes_{\mathcal{A}}{}_-$ we refer to e.g. \cite[Page~1139]{MOS}).
\end{proof}

\subsection{Image of a functor}\label{s1.6}

Let $\mathcal{A}$ and $\mathcal{C}$ be two categories and $\mathsf{F}:\mathcal{A}\to\mathcal{C}$ be a functor.
In general, one cannot speak about the ``image'' of $\mathsf{F}$. However, if both $\mathcal{A}$ and 
$\mathcal{C}$ are small and $\mathsf{F}$ induces an injective map from objects of $\mathcal{A}$ to objects
of $\mathcal{C}$, then we can define the {\em image} $\mathsf{F}(\mathcal{A})$ of $\mathsf{F}$ to be the 
subcategory of $\mathcal{C}$ which consists of all objects $\mathsf{F}(\mathtt{i})$ for $\mathtt{i}\in\mathcal{A}$ 
and all morphisms $\mathsf{F}(\alpha)$ where $\alpha$ is a morphism in $\mathcal{A}$.

\section{Finitary $2$-categories and their $2$-representations}\label{s2}

\subsection{Various $2$-categories}\label{s2.1}

In this paper by a $2$-category we mean a strict locally small $2$-category (see \cite{Le} for a concise 
introduction to $2$-categories and bicategories). Let $\cC$ be a $2$-category. 
We will use $\mathtt{i},\mathtt{j},\dots$  to denote objects in $\cC$;
$1$-morphisms in $\cC$ will be denoted by $\mathrm{F},\mathrm{G},\dots$; $2$-morphisms in $\cC$ will be denoted 
by $\alpha,\beta,\dots$. For $\mathtt{i}\in\cC$ we denote by $\mathbbm{1}_{\mathtt{i}}$ the corresponding
identity $1$-morphism. For a $1$-morphism $\mathrm{F}$ we denote by $\mathrm{id}_{\mathrm{F}}$ the 
corresponding identity $2$-morphism. We will write $\circ_0$ for horizontal composition of $2$-morphisms and $\circ_1$ for vertical composition of $2$-morphisms. The {\em opposite} $2$-category $\cC^{\mathrm{op}}$ is obtained
by reversing all $1$-morphisms and keeping the direction of all  $2$-morphisms.

Denote by $\mathbf{Cat}$ the $2$-category of all small categories. Let $\Bbbk$ be an algebraically closed field. 
Denote by $\mathfrak{A}_{\Bbbk}$ the $2$-category whose objects are small $\Bbbk$-linear fully additive categories, by which we mean additive categories which are idempotent-closed (or Karoubian);
$1$-morphisms are additive $\Bbbk$-linear functors and $2$-morphisms are natural transformations. Denote by 
$\mathfrak{A}_{\Bbbk}^f$ the full $2$-subcategory of $\mathfrak{A}_{\Bbbk}$ whose objects are 
fully additive categories $\mathcal{A}$ such that $\mathcal{A}$ has only finitely many 
isomorphism classes of indecomposable objects and all morphism spaces in $\mathcal{A}$ are finite dimensional.
We also denote by $\mathfrak{R}_{\Bbbk}$ the full subcategory of $\mathfrak{A}_{\Bbbk}$ containing all
objects which are equivalent to $A\text{-}\mathrm{mod}$ for some finite dimensional associative 
$\Bbbk$-algebra $A$.

\subsection{Finitary $2$-categories}\label{s2.2}

A $2$-category $\cC$ is called {\em finitary (over $\Bbbk$)}, see  \cite{MM}, if the following conditions are satisfied:
\begin{itemize}
\item $\cC$ has finitely many objects up to equivalence;
\item for any $\mathtt{i},\mathtt{j}\in\cC$ we have $\cC(\mathtt{i},\mathtt{j})\in \mathfrak{A}_{\Bbbk}^f$
and horizontal composition is both additive and $\Bbbk$-linear;
\item for any $\mathtt{i}\in\cC$ the $1$-morphism $\mathbbm{1}_{\mathtt{i}}$ is indecomposable.
\end{itemize}
%We will call  $\cC$  {\em weakly fiat} provided that it has a weak object preserving anti-autoequivalence $*$
%and for any $1$-morphism $\mathrm{F}\in\cC(\mathtt{i},\mathtt{j})$ there exist 
%$2$-morphisms $\alpha:\mathrm{F}\circ\mathrm{F}^*\to
%\mathbbm{1}_{\mathtt{j}}$ and $\beta:\mathbbm{1}_{\mathtt{i}}\to \mathrm{F}^*\circ\mathrm{F}$ such that 
%$\alpha_{\mathrm{F}}\circ_1\mathrm{F}(\beta)=\mathrm{id}_{\mathrm{F}}$ and
%$\mathrm{F}^*(\alpha)\circ_1\beta_{\mathrm{F}^*}=\mathrm{id}_{\mathrm{F}^*}$.
%If $*$ is involutive, then $\cC$ is called  {\em fiat}, see \cite{MM}.

\subsection{Homomorphisms, strong transformations and modifications}\label{s2.3}

Here we closely follow \cite{Le}.
Let $\cA$ and $\cC$ be $2$-categories. A {\em homomorphism} $\mathbf{H}:\cA\to \cC$ consists of the following data:
\begin{itemize}
\item a map $\mathbf{H}$ from objects of $\cA$ to objects of $\cC$;
\item functors $\mathbf{H}_{\mathtt{i},\mathtt{j}}:\cA(\mathtt{i},\mathtt{j})\to 
\cC(\mathbf{H}(\mathtt{i}),\mathbf{H}(\mathtt{j}))$;
\item natural isomorphisms
\begin{displaymath}
\mathbf{h}_{\mathrm{G},\mathrm{F}}:\mathbf{H}_{\mathtt{i},\mathtt{j}}(\mathrm{G})\circ 
\mathbf{H}_{\mathtt{i},\mathtt{j}}(\mathrm{F})\to \mathbf{H}_{\mathtt{i},\mathtt{j}}(\mathrm{G}\circ \mathrm{F})\quad
\text{ and }\quad
\mathbf{h}_{\mathtt{i}}:\mathbbm{1}_{\mathbf{H}(\mathtt{i})}\to \mathbf{H}_{\mathtt{i},\mathtt{i}}(\mathbbm{1}_{\mathtt{i}});
\end{displaymath}
\end{itemize}
such that the following conditions are satisfied:
\begin{displaymath}
\begin{array}{rcl}
\mathbf{h}_{\mathrm{H}\circ\mathrm{G},\mathrm{F}}\circ_1 
(\mathbf{h}_{\mathrm{H},\mathrm{G}}\circ_0 \mathrm{id}_{\mathbf{H}_{\mathtt{i},\mathtt{j}}(\mathrm{F})})&=&
\mathbf{h}_{\mathrm{H},\mathrm{G}\circ\mathrm{F}}\circ_1
(\mathrm{id}_{\mathbf{H}_{\mathtt{k},\mathtt{l}}(\mathrm{H})}\circ_0 \mathbf{h}_{\mathrm{G},\mathrm{F}}),\\
\mathrm{id}_{\mathbf{H}_{\mathtt{i},\mathtt{j}}(\mathrm{F})}&=&
\mathbf{h}_{\mathrm{F},\mathbbm{1}_{\mathtt{i}}}\circ_1 
(\mathrm{id}_{\mathbf{H}_{\mathtt{i},\mathtt{j}}(\mathrm{F})}\circ_0 \mathbf{h}_{\mathtt{i}}),\\
\mathrm{id}_{\mathbf{H}_{\mathtt{i},\mathtt{j}}(\mathrm{F})}&=&
\mathbf{h}_{\mathbbm{1}_{\mathtt{j}},\mathrm{F}}\circ_1 
(\mathbf{h}_{\mathtt{j}}\circ_0 \mathrm{id}_{\mathbf{H}_{\mathtt{i},\mathtt{j}}(\mathrm{F})}).
\end{array}
\end{displaymath}

Given homomorphisms $\mathbf{H}$ and $\mathbf{G}$ from $\cA$ to $\cC$ a {\em strong transformation}
$\Phi:\mathbf{H}\to \mathbf{G}$ is given by the following data:
\begin{itemize}
\item functors $\Phi_{\mathtt{i}}:\mathbf{H}(\mathtt{i})\to \mathbf{G}(\mathtt{i})$;
\item natural isomorphisms $\varphi_{\mathrm{F}}: \mathbf{G}(\mathrm{F})\circ \Phi_{\mathtt{i}}\to
\Phi_{\mathtt{j}}\circ \mathbf{H}(\mathrm{F})$;
\end{itemize}
such that the following conditions are satisfied:
\begin{displaymath}
\begin{array}{rcl}
\varphi_{\mathrm{G}\circ \mathrm{F}}\circ_1 (\mathbf{g}_{\mathrm{G},\mathrm{F}}\circ_0 
\mathrm{id}_{\Phi_{\mathtt{i}}})&=&(\mathrm{id}_{\Phi_{\mathtt{k}}}\circ_0\mathbf{h}_{\mathrm{G},\mathrm{F}})
\circ_1 (\varphi_{\mathrm{G}}\circ_0 \mathrm{id}_{\mathbf{H}(\mathrm{F})})\circ_1
(\mathrm{id}_{\mathbf{G}(\mathrm{G})}\circ_0 \varphi_{\mathrm{F}}); \\
\mathrm{id}_{\Phi_{\mathtt{i}}}\circ_0 \mathbf{h}_{\mathtt{i}}&=&\varphi_{\mathbbm{1}_{\mathtt{i}}}\circ_1
(\mathbf{g}_{\mathtt{i}}\circ_0\mathrm{id}_{\Phi_{\mathtt{i}}}).
\end{array}
\end{displaymath}
If the natural isomorphisms $\varphi_{\mathrm{F}}$ are identities, the strong transformation is called a
{\em strict transformation}.

Given two strong transformations $\Phi,\Psi:\mathbf{H}\to \mathbf{G}$ a {\em modification}
$\theta:\Phi\to\Psi$ is a collection of $2$-morphisms $\theta_{\mathtt{i}}:\Phi_{\mathtt{i}}\to \Psi_{\mathtt{i}}$
such that 
\begin{displaymath}
\psi_{\mathrm{F}}\circ_1 (\mathrm{id}_{\mathbf{G}(\mathrm{F})}\circ_0 \theta_{\mathtt{i}})=
(\theta_{\mathtt{j}}\circ_0 \mathrm{id}_{\mathbf{H}(\mathrm{F})})\circ_1\varphi_{\mathrm{F}}.
\end{displaymath}

\subsection{$2$-representations}\label{s2.4}

Let $\cC$ be a finitary $2$-category. We define the following $2$-categories of $2$-representations of $\cC$:
\begin{itemize}
\item the $2$-category $\cC\text{-}\mathrm{MOD}$ has as objects all homomorphisms from $\cC$ to
$\mathbf{Cat}$, as $1$-morphisms all strong transformations and as $2$-morphisms all modifications;
\item the $2$-category $\cC\text{-}\mathrm{amod}$ has as objects all homomorphisms from $\cC$ to
$\mathfrak{A}_{\Bbbk}$, as $1$-morphisms all strong transformations and as $2$-morphisms all modifications;
\item the $2$-category $\cC\text{-}\mathrm{afmod}$ has as objects all homomorphisms from $\cC$ to
$\mathfrak{A}_{\Bbbk}^f$, as $1$-morphisms all strong transformations and as $2$-morphisms all modifications;
\item the $2$-category $\cC\text{-}\mathrm{mod}$ has as objects all homomorphisms from $\cC$ to
$\mathfrak{R}_{\Bbbk}$, as $1$-mor\-phisms all strong transformations and as $2$-morphisms all modifications.
\end{itemize}
These are indeed $2$-categories since $\mathfrak{A}_{\Bbbk}$, $\mathfrak{A}_{\Bbbk}^f$ and $\mathfrak{R}_{\Bbbk}$ are.
We will write $\mathrm{Hom}_{\ccC}$ for $\mathrm{Hom}_{\ccC\text{-}\mathrm{afmod}}$.

\subsection{Biequivalence}\label{s2.5}

Let $\cA$ and $\cC$ be two $2$-categories. A {\em biequivalence} $\mathbf{H}:\cA\to\cC$ is a homomorphism
which is essentially surjective on objects and which is a {\em local equivalence}, that is
$\mathbf{H}_{\mathtt{i},\mathtt{j}}$ is an equivalence for all $\mathtt{i}$ and $\mathtt{j}$.
The $2$-categories $\cA$ and $\cC$ are called {\em biequivalent} if there is a biequivalence from
$\cA$ to $\cC$. Biequivalence is an equivalence relation, see \cite[Subsection~2.2]{Le}.

Alternatively, two homomorphisms $\mathbf{H}:\cA\to\cC$ and $\mathbf{G}:\cC\to\cA$ are mutually inverse biequivalences
if there exist strong transformations
\begin{equation}\label{eq1}
\mathbf{H}\circ \mathbf{G}\overset{\Phi_1}{\longrightarrow} \mathrm{Id}_{\ccC},\,\,\,\,
\mathrm{Id}_{\ccC}\overset{\Phi_2}{\longrightarrow} \mathbf{H}\circ \mathbf{G},\,\,\,\, 
\mathbf{G}\circ\mathbf{H}\overset{\Psi_1}{\longrightarrow} \mathrm{Id}_{\ccA}\quad\text{ and }\quad
\mathrm{Id}_{\ccA}\overset{\Psi_2}{\longrightarrow}  \mathbf{G}\circ\mathbf{H}
\end{equation}
and modifications
\begin{equation}\label{eq2}
\begin{array}{ccc}
\theta_{1}:\Phi_1\circ\Phi_2\to \mathbbm{1}_{\mathrm{Id}_{\ccC}},&\hspace{1cm}&
\theta_{2}:\mathbbm{1}_{\mathbf{H}\circ \mathbf{G}}\to \Phi_2\circ\Phi_1,\\
\theta_{3}:\Psi_1\circ\Psi_2\to \mathbbm{1}_{\mathrm{Id}_{\ccA}},&\hspace{1cm}&
\theta_{4}:\mathbbm{1}_{\mathbf{G}\circ \mathbf{H}}\to \Psi_2\circ\Psi_1.
\end{array}
\end{equation}
such that the latter are isomorphisms.

\subsection{Cancellative envelope of a $2$-category}\label{s2.6}

Let $\cC$ be a finitary $2$-category. Define a new $2$-category $\widehat{\cC}$ as follows:
\begin{itemize}
\item $\widehat{\cC}$ has the same objects as $\cC$;
\item $1$-morphisms in $\widehat{\cC}$ are all possible expressions of the form 
$\mathrm{F}_{\mathrm{X}_1,\dots,\mathrm{X}_k}$, where $k\in\mathbb{N}$ and $\mathrm{F}$, $\mathrm{X}_1$,\dots,
$\mathrm{X}_k$ are $1$-morphisms in $\cC$ such that $\mathrm{F}=\mathrm{X}_1\circ\dots\circ\mathrm{X}_k$, and also all
possible expressions of the form  $(\mathbbm{1}_{\mathtt{i}})_{\varnothing}$, where $\mathtt{i}\in\cC$;
\item $\widehat{\cC}(\mathrm{F}_{\mathrm{X}_1,\dots,\mathrm{X}_k},\mathrm{G}_{\mathrm{Y}_1,\dots,\mathrm{Y}_m}):=
{\cC}(\mathrm{F},\mathrm{G})$; 
\item the identity $1$-morphisms are $(\mathbbm{1}_{\mathtt{i}})_{\varnothing}$, where $\mathtt{i}\in\cC$;
\item horizontal composition of $1$-morphisms is defined via
\begin{displaymath}
\mathrm{F}_{\mathrm{X}_1,\dots,\mathrm{X}_k}\circ\mathrm{G}_{\mathrm{Y}_1,\dots,\mathrm{Y}_m}:=
(\mathrm{F}\circ\mathrm{G})_{\mathrm{X}_1,\dots,\mathrm{X}_k,\mathrm{Y}_1,\dots,\mathrm{Y}_m};
\end{displaymath}
\item both horizontal and vertical composition of $2$-morphisms are induced from the corresponding compositions in $\cC$.
\end{itemize}

Forgetting the subscripts of $1$-morphisms defines a $2$-functor from $\widehat{\cC}$ to $\cC$ which is a biequivalence
by construction. The category $\widehat{\cC}$ is {\em cancellative} in the sense that for any $1$-morphisms
$\mathrm{F}$, $\mathrm{G}$ and $\mathrm{H}$ in $\widehat{\cC}$ the equality 
$\mathrm{F}\circ\mathrm{H}=\mathrm{G}\circ\mathrm{H}$
implies $\mathrm{F}=\mathrm{G}$ and, moreover, the equality $\mathrm{H}\circ\mathrm{F}=\mathrm{H}\circ\mathrm{G}$
implies $\mathrm{F}=\mathrm{G}$.

\subsection{$2$-representations of biequivalent $2$-categories}\label{s2.7}

We will need the following observation.

\begin{proposition}\label{prop123}
Let $\cA$ and $\cC$ be two biequivalent finitary $2$-categories. Then 
the $2$-categories $\cA\text{-}\mathrm{afmod}$ and $\cC\text{-}\mathrm{afmod}$ are biequivalent.
\end{proposition}

\begin{proof}
Every homomorphism $\mathbf{H}:\cA\to\cC$ induces a homomorphism 
\begin{displaymath}
{}_-\circ\mathbf{H}: \cC\text{-}\mathrm{afmod}\to \cA\text{-}\mathrm{afmod}. 
\end{displaymath}
If $\mathbf{H}:\cA\to\cC$ and $\mathbf{G}:\cC\to\cA$ are inverse biequivalences, then 
we claim that ${}_-\circ\mathbf{H}$ and ${}_-\circ\mathbf{G}$ are also inverse biequivalences.
Given the data of \eqref{eq1} and \eqref{eq2}, composition with 
$\mathrm{Id}_{\ccC\text{-}\mathrm{afmod}}$ and $\mathrm{Id}_{\ccA\text{-}\mathrm{afmod}}$
defines the data establishing biequivalences between 
$\cA\text{-}\mathrm{afmod}$ and $\cC\text{-}\mathrm{afmod}$.
\end{proof}

Combining Proposition~\ref{prop123} with construction of the cancellative envelope in Subsection~\ref{s2.7},
 we may, without loss of generality, always assume that $\cC$ is cancellative.

\section{Projective $2$-representations}\label{s4}

\subsection{Principal $2$-representations}\label{s4.001}

Let $\cC$ be a finitary $2$-category. For any $\mathtt{i}\in \cC$ we have the {\em principal} 
(finitary) $2$-representation $\cC(\mathtt{i},{}_-)$ of $\cC$ which we denote by $\mathbb{P}_{\mathtt{i}}$.
For any $\mathbf{M}\in \cC\text{-}\mathrm{afmod}$ we have the Yoneda equivalence of categories 
\begin{equation}\label{eq3}
\mathrm{Hom}_{\ccC\text{-}\mathrm{afmod}}(\mathbb{P}_{\mathtt{i}},\mathbf{M})\cong \mathbf{M}(\mathtt{i})
\end{equation}
given by evaluation at $\mathbbm{1}_{\mathtt{i}}$, which is, moreover, surjective on objects
(the proof is analogous to \cite[Lemma~3]{MM3}). The direct sum of principal $2$-representations with every $\mathtt{i}\in \cC$ occurring exactly once is the $2$-analog of a ``free module of rank one'' in 
classical representation theory.

For $\mathbf{M}\in\cC\text{-}\mathrm{afmod}$ consider the $2$-category $\cM$ defined as follows:
\begin{itemize}
\item objects of $\cM$ are $(\eta,\Phi,\Phi',\mathbb{P})$ where $\mathbb{P}$ is a direct sum of principal 
$2$-representations, $\Phi,\Phi':\mathbb{P}\to \mathbf{M}$ are strong transformations and
$\eta:\Phi\to\Phi'$ is a modification;
\item a $1$-morphism from $(\eta,\Phi,\Phi',\mathbb{P})$ to $(\zeta,\Psi,\Psi',\mathbb{P}')$ is a strong 
transformation $\Lambda:\mathbb{P}\to\mathbb{P}'$ such that $\Psi\circ\Lambda=\Phi$,
$\Psi'\circ\Lambda=\Phi'$ and $\zeta\circ_0\mathrm{id}_{\Lambda}=\eta$;
\item a $2$-morphism of $\cM$ is just a modification $\lambda:\Lambda\to\Lambda'$ such that
$\zeta\circ_0\lambda=\eta$.
\end{itemize}
The $2$-category $\cM$ is the $2$-category of principal covers of $\mathbf{M}$, from which, as we will see below in  Proposition \ref{prop91}, we can recover $\mathbf{M}$ as a colimit. This will be important in obtaining a description of projective $2$-representations in Proposition \ref{prop92} and is to be thought of as a substitute for 
``having enough projective modules'' in classical representation theory.

Define a $2$-functor $\Upsilon$ from $\cM$ to $\cC\text{-}\mathrm{afmod}$ by sending 
$(\eta,\Phi,\Phi',\mathbb{P})$ to $\mathbb{P}$ and defining $\Upsilon$ as the identity on both
$1$-morphisms and $2$-morphisms. Recall that a (strict) cocone $(\mathbf{N},\mathfrak{W})$ of $\Upsilon$ 
is an object $\mathbf{N} \in \cC\text{-}\mathrm{afmod}$ together with, for any 
$(\eta,\Phi,\Phi',\mathbb{P})\in \cM$,  an assignment of two strong 
transformations $\Theta_1, \Theta_2:  \Upsilon(\eta,\Phi,\Phi',\mathbb{P}) \to \mathbf{N}$ and
a modification $\theta:\Theta_1\to\Theta_2$ i.e.
\begin{displaymath}
\xymatrix{
\mathfrak{W}_{(\eta,\Phi,\Phi',\mathbb{P})}=& \mathbb{P} \rrtwocell^{\Theta_1}_{\Theta_{2}}{\theta_\eta}&& \mathbf{N}
}
\end{displaymath} 
such that  for all $2$-morphisms 
$\lambda:\Lambda \to \Lambda'$ in $\cM$ as above, we have $\theta_\zeta\circ_0\lambda = \theta_\eta$.
A strict colimit of $\Upsilon$ is an initial object in the $2$-category of cocones of $\Upsilon$.
Clearly, the pair $(\mathbf{M},\mathfrak{V})$ defined  via 
$\mathfrak{V}_{(\eta,\Phi,\Phi',\mathbb{P})}:=(\Phi\overset{\eta}{\Longrightarrow}\Phi')$ is a strict cocone of  $\Upsilon$.

\begin{proposition}\label{prop91}
The cocone $(\mathbf{M},\mathfrak{V})$ is a strict colimit.
\end{proposition}

\begin{proof}
Let $(\mathbf{N},\mathfrak{W})$ be another strict cocone of $\Upsilon$ in $\cC\text{-}\mathrm{afmod}$.
To define a strong transformation $\Theta:\mathbf{M}\to\mathbf{N}$, let $\mathtt{i} \in \cC$, 
$X,Y\in \mathbf{M}(\mathtt{i})$ and $f:X\to Y$. Consider $(\eta,\Phi,\Phi',\mathbb{P}_{\mathtt{i}})$ such that 
$\Phi$ is a strict transformation sending $\mathbbm{1}_{\mathtt{i}}$ to $X$,
$\Phi'$ is a strict transformation sending $\mathbbm{1}_{\mathtt{i}}$ to $Y$
and $\eta$ is a modification which evaluates to $f$ (these exist by the Yoneda lemma). Due to strictness of $(\mathbf{N},\mathfrak{W})$ as a cocone, we
necessarily have $\Theta(X\overset{f}{\longrightarrow}Y)=
\mathfrak{W}_{(\eta,\Phi,\Phi',\mathbb{P}_\mathtt{i})}(\mathbbm{1}_{\mathtt{i}})$. On the other hand,
the latter defines a morphism from $(\mathbf{M},\mathfrak{V})$ to $(\mathbf{N},\mathfrak{W})$. The claim follows.
\end{proof}

\subsection{Projective $2$-representations}\label{s4.002}

Let $\cC$ be a finitary $2$-category.  A finitary $2$-representation $\mathbf{P}$ of $\cC$ is called
{\em projective} if $\mathrm{Hom}_{\ccC}(\mathbf{P},{}_-)$ preserves all small colimits.

\begin{proposition}\label{prop92} 
A finitary $2$-representation $\mathbf{P}$ of $\cC$ is projective if and only if it is a retract
of a direct sum of principal $2$-representations.
\end{proposition}

\begin{proof}
We follow the classical argument from \cite[Proposition~2]{BD}.  From the Yoneda lemma it follows that all 
principal representations (and their direct sums) are projective. Let $\mathbb{P}$ be a direct sum
of principal $2$-representations and $\mathbf{P}$ a retract of $\mathbb{P}$, that is there exist 
$\Phi:\mathbb{P}\to \mathbf{P}$ and $\Psi:\mathbf{P}\to \mathbb{P}$ such that $\Phi\Psi$ is isomorphic to the
identity on $\mathbf{P}$. Then $\mathrm{Hom}_{\ccC}(\mathbf{P},{}_-)$ is isomorphic to 
$\mathrm{Hom}_{\ccC}(\mathbb{P},{}_-)\circ \Psi$ and hence commutes with all small colimits (as
$\mathrm{Hom}_{\ccC}(\mathbb{P},{}_-)$ does).

Let now $\mathbf{M}$ be a projective finitary $2$-representation of $\cC$. By Proposition~\ref{prop91},
$\mathbf{M}$ is equivalent to the colimit of $\Upsilon:\cM\to\cC\text{-}\mathrm{afmod}$. 
Since $\mathrm{Hom}_{\ccC}(\mathbf{M},{}_-)$ 
preserves all small colimits, we have
\begin{displaymath}
\mathrm{Hom}_{\ccC}(\mathbf{M},\mathbf{M})\cong
\mathrm{Hom}_{\ccC}(\mathbf{M},\lim_{\to}\Upsilon)\cong
\lim_{\to}\mathrm{Hom}_{\ccC}(\mathbf{M},\Upsilon).
\end{displaymath}
Via this equivalence, the identity on $\mathbf{M}$ thus must have a representative in the right hand side. 
This means that there exists a direct sum $\mathbb{P}$ of principal $2$-representations and
$\Phi:\mathbb{P}\to \mathbf{M}$ such that the identity on $\mathbf{M}$ is represented in 
the term indexed by $(\mathrm{id}_{\Phi},\Phi,\Phi,\mathbb{P})$, say by some $\Psi$. This means that
$\Phi\Psi$ is the identity on $\mathbf{M}$. The claim follows.
\end{proof}

\subsection{Idempotent matrices with non-negative integer coefficients}\label{s4.2}

For an arbitrary $n\in\mathbb{N}$ we denote by $\mathtt{0}_n$ the zero $n\times n$ matrix and by
$\mathtt{1}_n$ the identity $n\times n$ matrix. We will need the following result from \cite[Theorem~2]{Fl}:

\begin{proposition}\label{prop401}
Let $M$ be an idempotent matrix with non-negative integer coefficients. Then there is a permutation
matrix $S$ such that $S^{-1}MS$ has the form
\begin{equation}\label{eq77}
\left(\begin{array}{ccc}\mathtt{0}_a&A&AB\\\mathtt{0}&\mathtt{1}_b&B\\
\mathtt{0}&\mathtt{0}&\mathtt{0}_c\end{array}\right) 
\end{equation}
for some matrices $A$ and $B$.
\end{proposition}

\subsection{Idempotent endomorphisms of principal representations}\label{s4.3}

By \eqref{eq3}, for every $\mathtt{i},\mathtt{j}\in\cC$ we have 
$\mathrm{Hom}_{\ccC\text{-}\mathrm{afmod}}(\mathbb{P}_{\mathtt{i}},\mathbb{P}_{\mathtt{j}})
\cong \cC(\mathtt{j},\mathtt{i})$, which means that every homomorphism from  
$\mathbb{P}_{\mathtt{i}}$ to $\mathbb{P}_{\mathtt{j}}$ is isomorphic to right multiplication by
some $1$-morphism in $\cC(\mathtt{j},\mathtt{i})$. Similarly, given a finite direct sum of 
principal $2$-representations (with, say, $k$ summands), every endomorphism $\Phi$ of this direct sum is
isomorphic to right multiplication by a $k\times k$ matrix whose coefficients are appropriate
$1$-morphisms in $\cC$. We will call a summand of any entry in this matrix a {\em summand} of $\Phi$.

Let $\displaystyle\mathbf{P}=\bigoplus_{s=1}^k \mathbb{P}_{\mathtt{i}_s}$ be a finite direct sum of 
principal $2$-representations of $\cC$ and $\Phi\in\mathrm{End}_{\ccC\text{-}\mathrm{afmod}}(\mathbf{P})$. 
We associate with $\Phi$ a matrix $M:=M_{\Phi}$ defined as follows: the matrix $M$ is a block 
$k\times k$ matrix with blocks indexed by the summands of $\mathbf{P}$. The rows of the $(r,s)$-block 
(i.e. block-row index $r$ and  block-column index $s$) are indexed by indecomposable $1$-morphisms
in $\cC(\mathtt{i}_r,\mathtt{j})$ for any $\mathtt{j}$ in $\cC$. The columns of the 
$(r,s)$-block are indexed by indecomposable 
$1$-morphisms in $\cC(\mathtt{i}_s,\mathtt{j})$ for any $\mathtt{j}$ in $\cC$. Let $\mathtt{F}$ be an indecomposable 
$1$-morphism in $\cC(\mathtt{i}_r,\mathtt{j})$ and $\mathtt{G}$ be an indecomposable 
$1$-morphism in $\cC(\mathtt{i}_s,\mathtt{j}')$. Then the $(\mathtt{F},\mathtt{G})$-entry in the 
$(r,s)$-block is given by the multiplicity of $\mathtt{G}$ as a direct summand of  $\Phi(\mathtt{F})$
(in particular, if this entry is nonzero, then $\mathtt{j}=\mathtt{j}'$).
Note that $M$ is a square matrix with non-negative integer coefficients. Furthermore, if
$\Psi\in\mathrm{End}_{\ccC\text{-}\mathrm{afmod}}(\mathbf{P})$, then 
$M_{\Phi}M_{\Psi}=M_{\Psi\circ \Phi}$.  In particular, $\Psi$  is idempotent,
by which we mean $\Psi^2\cong\Psi$,  if and only if we have $M_{\Psi}^2=M_{\Psi}$.

Assume that $\Phi\cong\Phi^2\neq 0$, then $M\neq 0$. By Proposition~\ref{prop401}, in this case all diagonal entries of
$M$ are equal to either $0$ or $1$ and there is at least one non-zero diagonal entry. 
Let $\mathrm{F}_1,\mathrm{F}_2,\dots,\mathrm{F}_m$ be a complete list of indecomposable $1$-morphisms indexing the
non-zero diagonal entries of $M$. 

\begin{lemma}\label{lem403}
\begin{enumerate}[$($a$)$]
\item\label{lem403.1} For every $i\in\{1,2,\dots,m\}$, let $\mathtt{j}_i, \mathtt{k}_i$ be such that $\mathrm{F}_i \in \cC(\mathtt{j}_i,\mathtt{k}_i)$. Then there is a unique indecomposable summand
$\Gamma_i$ of $\Phi$, given by right multiplication by an indecomposable 
$1$-morphism $\mathrm{G}_i \in \cC(\mathtt{j}_i,\mathtt{j}_i)$, and such that 
$\Gamma_i(\mathrm{F}_i)\cong \mathrm{F}_i\circ \mathrm{G}_i\cong\mathrm{F}_i\oplus \mathrm{X}_i$
for some $1$-morphism $\mathrm{X}_i \in \cC(\mathtt{j}_i,\mathtt{k}_i)$. 
\item\label{lem403.2} We have $\Phi(\mathrm{X})\cong \mathrm{X}\circ \mathrm{G}_i=0$.
\end{enumerate}
\end{lemma}

\begin{proof}We use Proposition~\ref{prop401} to reduce $M$ to the form 
$\tilde{M}$ given by \eqref{eq77}. Then the multiplicities of $\Phi(\mathrm{F}_i)$ are given 
by the row $v$ of $\tilde{M}$ indexed by $\mathrm{F}_i$. Note that $\mathrm{F}_i$ indexes a row in the second row of blocks
of the $3\times 3$ block decomposition of $\tilde{M}$. Therefore $\Phi(\mathrm{F}_i) = \mathrm{F}_i \oplus \mathrm{Y}$ for some  $\mathrm{Y}$ with $\Phi(\mathrm{Y})=0$. Now there is a unique indecomposable summand $\Gamma_i$  of $\Phi$ which contributes the summand $\mathrm{F}_i$ above, and it is given by right multiplication with a unique indecomposable $1$-mor\-phism $\mathrm{G}_i \in \cC(\mathtt{j}_i,\mathtt{j}_i)$. Define $\mathrm{X}_i$ via $\mathrm{F}_i\circ\mathrm{G}_i\cong\mathrm{F}_i\oplus \mathrm{X}_i$. Then, with $\mathrm{X}_i$ being a summand of $\mathrm{Y}$ and $\Gamma_i$ being a summand of $\Phi$, we have that  $\Gamma_i(\mathrm{X}_i)$ is a summand of $\Phi(\mathrm{Y})$ and hence equals zero.
\end{proof}

\begin{lemma}\label{lem404}
For $i\in\{1,2,\dots,m\}$ the $1$-morphism $\mathrm{G}_i$ satisfies
$\mathrm{G}_i\circ \mathrm{G}_i\cong \mathrm{G}_i\oplus \mathrm{Q}_i$ such that 
$\mathrm{G}_i\circ \mathrm{Q}_i=0$ and $\Phi(\mathrm{Q}_i)=0$.
\end{lemma}

\begin{proof}
From Lemma~\ref{lem403}\eqref{lem403.1} we have that $\mathrm{G}_i$ is the only indecomposable summand of $\Phi$ sending $\mathrm{F}_i$ 
to $\mathrm{F}_i$ (plus something). This yields that $\mathrm{G}_i$ is a summand of $\mathrm{G}_i\circ \mathrm{G}_i$
occurring with multiplicity one (since $\mathrm{F}_i$ appears with multiplicity one in 
$\mathrm{F}_i\circ \mathrm{G}_i$). In particular, it follows that $\mathrm{G}_i\cong\mathrm{F}_{j_i}$
for some $j_i\in \{1,2,\dots,m\}$. We also have $\mathrm{G}_{j_i} = \mathrm{G}_i$, so that 
$\mathrm{Q}_i = \mathrm{X}_{j_i}$. From Lemma~\ref{lem403}\eqref{lem403.2} we thus get 
$\Phi(\mathrm{Q}_i)=0$, in particular $\mathrm{Q}_i\circ\mathrm{G}_i=0$. 
To prove $\mathrm{G}_i\circ \mathrm{Q}_i=0$ we compute
$\mathrm{G}_i^3$ in two different ways. On the one hand,
\begin{displaymath}
\mathrm{G}_i^3\cong \mathrm{G}_i\circ(\mathrm{G}_i\oplus \mathrm{Q}_i)\cong
\mathrm{G}_i\oplus \mathrm{Q}_i\oplus \mathrm{G}_i\circ\mathrm{Q}_i.
\end{displaymath}
On the other hand,
\begin{displaymath}
\mathrm{G}_i^3\cong (\mathrm{G}_i\oplus \mathrm{Q}_i)\circ\mathrm{G}_i\cong
\mathrm{G}_i\oplus \mathrm{Q}_i\oplus \mathrm{Q}_i\circ\mathrm{G}_i.
\end{displaymath}
Now $\mathrm{G}_i\circ \mathrm{Q}_i=0$ follows from $\mathrm{Q}_i\circ \mathrm{G}_i=0$ by comparing the
two isomorphisms above.
\end{proof}

\begin{lemma}\label{lem405}
Let $i,j\in \{1,2,\dots,m\}$ be such that $\mathrm{G}_i\not\cong\mathrm{G}_j$.
\begin{enumerate}[$($a$)$]
\item\label{lem405.1} We have $\Phi(\mathrm{G}_i\circ \mathrm{G}_j)=0$. 
\item\label{lem405.2} We have $\mathrm{G}_i\circ \mathrm{G}_j=0$.
\end{enumerate}
\end{lemma}

\begin{proof}
We use Proposition~\ref{prop401} to reduce $M$ to the form 
$\tilde{M}$ given by \eqref{eq77}. Note that $\mathrm{G}_i$ indexes a row in the second block
of the $3\times 3$ block decomposition of $\tilde{M}$, so $\Phi(\mathrm{G}_i)=\mathrm{G}_i \oplus \mathrm{Y}$ with $\Phi(\mathrm{Y})=0$. 
The composition $\mathrm{G}_i\circ \mathrm{G}_j$ is a direct summand of $\Phi(\mathrm{G}_i)$. Since
$\mathrm{G}_i\not\cong\mathrm{G}_j$, the composition $\mathrm{G}_i\circ \mathrm{G}_j$ does not contain
$\mathrm{G}_i$ as a direct summand (by Lemma~\ref{lem403}\eqref{lem403.1}) and is hence contained in $\mathrm{Y}$, implying that $\Phi(\mathrm{G}_i\circ \mathrm{G}_j)=0$. This proves claim \eqref{lem405.1}.

As $(\mathrm{G}_i\circ \mathrm{G}_j)\circ \mathrm{G}_j$ is a summand of $\Phi(\mathrm{G}_i\circ \mathrm{G}_j)$,
we obtain $(\mathrm{G}_i\circ \mathrm{G}_j)\circ \mathrm{G}_j=0$ by \eqref{lem405.1}. On the other hand,
$\mathrm{G}_i\circ (\mathrm{G}_j\circ \mathrm{G}_j)$ contains, as a summand, $\mathrm{G}_i\circ \mathrm{G}_j$
by Lemma~\ref{lem404}. This implies claim \eqref{lem405.2} and completes the proof.
\end{proof}

Note that it is possible that $\mathrm{G}_i\cong \mathrm{G}_j$ for $i\neq j$, where $i,j\in\{1,2,\dots,m\}$.
Therefore we define $\Gamma$ to be the multiplicity free direct sum of all $\Gamma_i$ for $i\in\{1,2,\dots,m\}$.
Similarly, define $\Theta$ such that $\Gamma^2\cong \Gamma\oplus \Theta$. From the above discussion it follows
that the endomorphism $\Theta$ is given by putting, for each $i$, a copy of $\mathrm{Q}_i$ 
(as given by Lemma~\ref{lem404}) in the appropriate places.
Define $\Pi$ to be a summand of $\Phi$ such that $\Phi=\Gamma\oplus\Theta\oplus\Pi$. As an immediate
corollary from  Lemma~\ref{lem404} we have.

\begin{corollary}\label{lem406}
We have $\Phi\Theta=0$.
\end{corollary}

\begin{lemma}\label{lem409}
We have the following identities:
\begin{enumerate}[$($a$)$]
\item\label{lem409.1} $\Gamma\Pi\Gamma=\Theta\Pi=\Gamma\Pi^2=\Pi^2\Gamma=0$. 
\item\label{lem409.2} $\Pi\cong\Gamma\Pi\oplus\Pi\Gamma\oplus\Pi^2$. 
\item\label{lem409.3} $\Pi^3=0$. 
\item\label{lem409.4} $\Pi^2\cong\Pi\Gamma\Pi$. 
\end{enumerate}
\end{lemma}

\begin{proof}
Using Corollary~\ref{lem406}, we compute
\begin{displaymath}
\Gamma\oplus\Theta\oplus\Pi=\Phi\cong\Phi^2=\Gamma\oplus\Theta\oplus \Gamma\Pi\oplus\Pi\Gamma\oplus\Theta\Pi\oplus\Pi^2
\end{displaymath}
which implies
\begin{equation}\label{e4091}
\Pi\cong\Gamma\Pi\oplus\Pi\Gamma\oplus\Theta\Pi\oplus\Pi^2.
\end{equation}
Inserting the right hand side of \eqref{e4091} into the first summand on the right we obtain
\begin{equation}\label{e4092}
\Pi\cong\Gamma(\Gamma\Pi\oplus\Pi\Gamma\oplus\Theta\Pi\oplus\Pi^2)\oplus\Pi\Gamma\oplus\Theta\Pi\oplus\Pi^2. 
\end{equation}
Using Corollary~\ref{lem406} and comparing the right hand sides of \eqref{e4091} and \eqref{e4092} gives the identities
$\Gamma\Pi\Gamma=\Gamma\Pi\Theta=\Gamma\Pi^2=0$.
Inserting the right hand side of \eqref{e4091} into the second summand on the right we obtain
\begin{equation}\label{e4093}
\Pi\cong\Gamma\Pi\oplus(\Gamma\Pi\oplus\Pi\Gamma\oplus\Theta\Pi\oplus\Pi^2)\Gamma\oplus\Theta\Pi\oplus\Pi^2. 
\end{equation}
Using Corollary~\ref{lem406} and comparing the right hand sides of \eqref{e4091} and \eqref{e4093} gives the identities
$\Theta\Pi=\Pi^2\Gamma=0$. This proves claim \eqref{lem409.1}. Claim \eqref{lem409.2} follows from \eqref{e4091}
and claim \eqref{lem409.1}.

Let $N$ be the matrix associated to $\Pi$ (similarly to how $M$ is associated to $\Phi$). Then both
$N$ and $M-N$ have non-negative integer coefficients. Let $S$ be a permutation matrix such that 
$S^{-1}MS=\tilde{M}$ and set $\tilde{N}=S^{-1}NS$. Then both $\tilde{N}$ and $\tilde{M}-\tilde{N}$ have 
non-negative integer coefficients and from the definition of $\Pi$ it follows that $\tilde{N}$ has the form
\begin{displaymath}
\left(\begin{array}{ccc}\mathtt{0}_a&A'&B'\\\mathtt{0}&\mathtt{0}_b&C'\\
\mathtt{0}&\mathtt{0}&\mathtt{0}_c\end{array}\right). 
\end{displaymath}
Clearly, $\tilde{N}^3=0$ and thus $\Pi^3=0$ proving claim \eqref{lem409.3}.

Inserting \eqref{lem409.2} into one of the factors in $\Pi^2=\Pi^2$,  we obtain
\begin{equation}\label{e4094}
\Pi^2\cong(\Gamma\Pi\oplus\Pi\Gamma\oplus\Pi^2)\Pi. 
\end{equation}
Now claim \eqref{lem409.4} follows from claims \eqref{lem409.1} and \eqref{lem409.3}. This completes the proof.
\end{proof}

\begin{corollary}\label{cor410}
We have the following isomorphisms:
\begin{enumerate}[$($a$)$]
\item\label{lem410.1} $(\Gamma\oplus\Theta)^2\cong\Gamma\oplus\Theta$. 
\item\label{lem410.2} $(\Gamma\oplus\Theta\oplus\Pi\Gamma)^2\cong\Gamma\oplus\Theta\oplus\Pi\Gamma$. 
\item\label{lem410.3} $(\Gamma\oplus\Theta\oplus\Pi\Gamma)(\Gamma\oplus\Theta)\cong\Gamma\oplus\Theta\oplus\Pi\Gamma$. 
\item\label{lem410.4} $(\Gamma\oplus\Theta)(\Gamma\oplus\Theta\oplus\Pi\Gamma)\cong\Gamma\oplus\Theta$. 
\item\label{lem410.5} $\Phi(\Gamma\oplus\Theta\oplus\Pi\Gamma)\cong\Gamma\oplus\Theta\oplus\Pi\Gamma$. 
\item\label{lem410.7} $(\Gamma\oplus\Theta\oplus\Pi\Gamma)\Phi\cong\Phi$. 
\end{enumerate} 
\end{corollary}

\begin{proof}
Claim \eqref{lem410.1} follows from the definitions and Corollary~\ref{lem406}.
Claim \eqref{lem410.2} follows from claim \eqref{lem410.1}, Corollary~\ref{lem406} and Lemma~\ref{lem409}\eqref{lem409.1}.
Claim \eqref{lem410.3} follows from Corollary~\ref{lem406}.
Claims \eqref{lem410.4} and \eqref{lem410.5} follow from Corollary~\ref{lem406} and Lemma~\ref{lem409}\eqref{lem409.1}.
Finally, claim \eqref{lem410.7} follows from Corollary~\ref{lem406} and
Lemma~\ref{lem409}\eqref{lem409.1}, \eqref{lem409.2} and \eqref{lem409.4}.
\end{proof}

\subsection{Projective $2$-subrepresentations of principal $2$-representations}\label{s4.4}

For $\mathtt{i}\in\cC$ fix an indecomposable $1$-morphism $\mathrm{G}\in \cC(\mathtt{i},\mathtt{i})$
and also a $1$-morphism $\mathrm{Q}\in \cC(\mathtt{i},\mathtt{i})$ such that 
\begin{displaymath}
\mathrm{G}\circ\mathrm{G}\cong \mathrm{G}\oplus \mathrm{Q}\quad\text{ and }\quad 
\mathrm{G}\circ\mathrm{Q}=\mathrm{Q}\circ\mathrm{G}=\mathrm{Q}\circ\mathrm{Q}=0.
\end{displaymath}
Then $\mathrm{E}:=\mathrm{G}\circ\mathrm{G}$ is a weakly idempotent $1$-morphism in $\cC(\mathtt{i},\mathtt{i})$
(in the sense that $\mathrm{E}\circ\mathrm{E}\cong\mathrm{E}$).
Assume now that $\cC$ is cancellative. In this case the  
functor $({}_-\circ\mathrm{E})_{\mathtt{j}}$ is injective when restricted to objects of
$\mathbb{P}_{\mathtt{i}}(\mathtt{j})$. Hence we can consider the corresponding image $\mathbb{P}_{\mathtt{i}}(\mathtt{j})\circ\mathrm{E}$.
The left action of $\cC$ leaves $\mathbb{P}_{\mathtt{i}}({}_-)\circ\mathrm{E}$ invariant. We denote by
$\mathbf{P}_{\mathtt{i},\mathrm{E}}$ this $2$-subrepresentation of $\mathbb{P}_{\mathtt{i}}$
and by $\Lambda$ the corresponding natural inclusion. Denote by 
$\Lambda':\mathbb{P}_{\mathtt{i}}\to\mathbf{P}_{\mathtt{i},\mathrm{E}}$ the strict transformation
given by sending $\mathbbm{1}_{\mathtt{i}}$ to $\mathrm{E}$.

\begin{proposition}\label{prop441}
The composition $\Lambda'\Lambda$ is isomorphic to the identity on $\mathbf{P}_{\mathtt{i},\mathrm{E}}$.
In particular, $\mathbf{P}_{\mathtt{i},\mathrm{E}}$ is projective.
\end{proposition}
\begin{proof}
 
Let $\mathrm{F}$ and $\mathrm{F}'$ be two $1$-morphisms in $\cC$. Then 
$\mathrm{F}\circ\mathrm{E}$ and $\mathrm{F}'\circ\mathrm{E}$ are in
$\mathbf{P}_{\mathtt{i},\mathrm{E}}$. 
Any morphism 
$\alpha: \mathrm{F}\circ\mathrm{E}\to \mathrm{F}'\circ\mathrm{E}$
in $\mathbf{P}_{\mathtt{i},\mathrm{E}}$ is, by definition, of the form $\beta\circ_0\mathrm{id}_{\mathrm{E}}$
for some $\beta:\mathrm{F}\to \mathrm{F}'$. The composition  $\Lambda'\Lambda$ is given by right multiplication by $\mathrm{E}$. Denote by 
$\eta: \mathrm{E}\to \mathrm{E}^2$ an isomorphism. Then the diagram
\begin{displaymath}
\xymatrix{ 
\mathrm{F}\circ\mathrm{E}\ar[rrrr]^{\beta\circ_0\mathrm{id}_{\mathrm{E}}}
\ar[d]_{\mathrm{id}_{\mathrm{F}}\circ_0\eta}
&&&& \mathrm{F}'\circ\mathrm{E}\ar[d]^{\mathrm{id}_{\mathrm{F}'}\circ_0\eta}\\
\mathrm{F}\circ\mathrm{E}\circ\mathrm{E}
\ar[rrrr]^{\beta\circ_0\mathrm{id}_{\mathrm{E}}\circ_0\mathrm{id}_{\mathrm{E}} }
&&&& \mathrm{F}'\circ\mathrm{E}\circ\mathrm{E}
}
\end{displaymath}
where the vertical arrows are isomorphisms provides an isomorphism from the identity on $\mathbf{P}_{\mathtt{i},\mathrm{E}}$ to $\Lambda'\Lambda$.
\end{proof}

\begin{corollary}\label{cor437}
\begin{enumerate}[$($a$)$]
\item\label{cor437.1} The restriction of ${}_-\circ\mathrm{E}$ to  $\mathbf{P}_{\mathtt{i},\mathrm{E}}$ 
is isomorphic to the identity functor on $\mathbf{P}_{\mathtt{i},\mathrm{E}}$.
\item\label{cor437.2} The restriction of ${}_-\circ\mathrm{G}$ to  $\mathbf{P}_{\mathtt{i},\mathrm{E}}$ 
is isomorphic to the identity functor on $\mathbf{P}_{\mathtt{i},\mathrm{E}}$.
\item\label{cor437.3} $\mathrm{End}_{\ccC}(\mathbf{P}_{\mathtt{i},\mathrm{E}})$ is biequivalent
to $\mathrm{E}\circ\cC(\mathtt{i},\mathtt{i})\circ\mathrm{E}$.
\end{enumerate}
\end{corollary}

\begin{proof}
Claim \eqref{cor437.1} is a direct consequence of  Proposition~\ref{prop441}.
Claim \eqref{cor437.2} follows from claim \eqref{cor437.1} since $\mathrm{E}\circ\mathrm{Q}=0$ 
(see Lemma~\ref{lem404}).

To prove claim \eqref{cor437.3} we consider the composition
\begin{displaymath}
\mathbb{P}_{\mathtt{i}}\overset{\Lambda'}{\longrightarrow} 
\mathbf{P}_{\mathtt{i},\mathrm{E}}\overset{\Phi}{\longrightarrow} \mathbf{P}_{\mathtt{i},\mathrm{E}}
\overset{\Lambda}{\longrightarrow}  \mathbb{P}_{\mathtt{i}} 
\end{displaymath}
where $\Phi\in \mathrm{End}_{\ccC}(\mathbf{P}_{\mathtt{i},\mathrm{E}})$.
Then $\Lambda\Phi\Lambda'$ is an endomorphism of $\mathbb{P}_{\mathtt{i}}$ and hence is given (up to equivalence)
by right multiplication with some $1$-morphism $\mathrm{F}\in \cC(\mathtt{i},\mathtt{i})$ by the Yoneda Lemma. 
We have $\mathrm{F}\circ\mathrm{E}\cong \mathrm{F}$ since $\mathrm{F}\in \mathbf{P}_{\mathtt{i},\mathrm{E}}$
by claim \eqref{cor437.1}. Further, using idempotency of $\mathrm{E}$ we have
\begin{displaymath}
\mathrm{F}\cong\Phi(\mathrm{E})\cong\Phi(\mathrm{E}\circ\mathrm{E}) 
 \cong\mathrm{E}\circ\Phi(\mathrm{E})\cong\mathrm{E}\circ\mathrm{F}.
\end{displaymath}
This yields that $\mathrm{F}$ is isomorphic to a $1$-morphism in
$\mathrm{E}\circ\cC(\mathtt{i},\mathtt{i})\circ\mathrm{E}$. If 
$\Phi'\in \mathrm{End}_{\ccC}(\mathbf{P}_{\mathtt{i},\mathrm{E}})$ is similarly given by some $\mathrm{F}'$
and $\eta: \Phi\to\Phi'$ is a modification, then, again by the Yoneda Lemma, the corresponding modification
\begin{displaymath}
\mathrm{id}_{\Lambda}\circ_0\eta\circ_0\mathrm{id}_{\Lambda'}:\Lambda\Phi\Lambda'\to\Lambda\Phi'\Lambda' 
\end{displaymath}
is given by some $2$-morphism $\alpha:\mathrm{F}\to\mathrm{F}'$. It follows that $\eta$ is given by 
restriction of $\alpha$ to $\mathbf{P}_{\mathtt{i},\mathrm{E}}$. Claim \eqref{cor437.3} follows.
\end{proof}

\subsection{Description of projective finitary $2$-representations}\label{s4.5}

Now we are ready to describe projective $2$-representations of $\cC$. We assume that $\cC$ is cancellative.

\begin{theorem}\label{thm412}
Let $\mathbf{P}$ be a projective  $2$-representation of $\cC$. Then
\begin{equation}\label{eq41201}
\mathbf{P}\cong \bigoplus_{s=1}^m \mathbf{P}_{\mathtt{i}_s,\mathrm{E}_s} 
\end{equation}
for some $\mathtt{i}_s\in\cC$ and $\mathrm{E}_s\in\cC(\mathtt{i}_s,\mathtt{i}_s)$ such that
$\mathrm{E}_s=\mathrm{G}_s\circ \mathrm{G}_s$ for some indecomposable
$\mathrm{G}_s\in\cC(\mathtt{i}_s,\mathtt{i}_s)$  and,
additionally,
\begin{displaymath}
\mathrm{G}_s\circ\mathrm{G}_s\cong \mathrm{G}_s\oplus \mathrm{Q}_s\quad\text{ and }\quad 
\mathrm{G}_s\circ\mathrm{Q}_s=\mathrm{Q}_s\circ\mathrm{G}=\mathrm{Q}_s\circ\mathrm{Q}_s=0
\end{displaymath} 
for some $\mathrm{Q}_s\in\cC(\mathtt{i}_s,\mathtt{i}_s)$.
\end{theorem}

\begin{proof}
Let $\mathbf{P}$ be a projective  $2$-representation of $\cC$. By Proposition~\ref{prop92},
there exists a direct sum $\mathbb{P}$ of principal $2$-representations and strong transformations
$\Phi:\mathbb{P}\to \mathbf{P}$ and $\Psi:\mathbf{P}\to\mathbb{P}$ 
such that $\Phi\Psi$ is isomorphic to the identity endomorphism of $\mathbf{P}$. 
Therefore $\Psi\Phi$ is isomorphic to an idempotent endomorphism of $\mathbb{P}$.

By our analysis in Subsection \ref{s4.3},  $\Psi\Phi$ is of the form $\Psi\Phi\cong \Gamma \oplus \Theta \oplus\Pi$ with 
$\Gamma^2 \cong\Gamma \oplus \Theta$ and $\Theta\Psi\Phi=\Psi\Phi\Theta=0$. By Corollary \ref{cor410}\eqref{lem410.1} 
and \eqref{lem410.3},  $\Gamma \oplus \Theta$ and $(\Gamma\oplus\Theta\oplus\Pi\Gamma)$ are also idempotent endomorphisms 
of $\mathbb{P}$. Moreover, by Corollary \ref{cor410}\eqref{lem410.5} and  \eqref{lem410.7}, we have 
\begin{equation}\label{eq425}
\begin{split}
\Psi\Phi &\cong \Psi\Phi(\Gamma\oplus\Theta\oplus\Pi\Gamma)\Psi\Phi \\
\Gamma\oplus\Theta\oplus\Pi\Gamma&\cong (\Gamma\oplus\Theta\oplus\Pi\Gamma)\Psi\Phi (\Gamma\oplus\Theta\oplus\Pi\Gamma)
\end{split}
\end{equation}
and similarly Corollary \ref{cor410}\eqref{lem410.3} and \eqref{lem410.4} yield
\begin{equation}\label{eq426}
\begin{split}
\Gamma\oplus\Theta\oplus\Pi\Gamma &\cong 
(\Gamma\oplus\Theta\oplus\Pi\Gamma)(\Gamma \oplus \Theta)(\Gamma\oplus\Theta\oplus\Pi\Gamma)  \\
\Gamma \oplus \Theta &\cong (\Gamma \oplus \Theta)(\Gamma\oplus\Theta\oplus\Pi\Gamma)(\Gamma \oplus \Theta).
\end{split}
\end{equation}

Consider the strong transformations
\begin{displaymath}
\begin{array}{lcrlll}
\Omega&:=&(\Gamma\oplus\Theta)(\Gamma\oplus\Theta\oplus\Pi\Gamma)\Psi:&\mathbf{P}&\to&\mathbb{P}\\
\Omega'&:=&\Phi(\Gamma\oplus\Theta\oplus\Pi\Gamma)(\Gamma\oplus\Theta):&\mathbb{P}&\to&\mathbf{P}.\\
\end{array}
\end{displaymath}
We have $\Omega'\Omega\cong \Phi\Psi \Omega'\Omega \Phi\Psi$ and the latter is isomorphic to the
identity on $\mathbf{P}$ by \eqref{eq425} and \eqref{eq426}. In the other direction, we deduce 
from \eqref{eq425} and \eqref{eq426} that the composition $\Omega\Omega'$ is isomorphic to 
$\Gamma \oplus \Theta$. The latter is a diagonal idempotent endomorphism of $\mathbb{P}$. Since $\cC$ is cancellative,
we have a well-defined image of $\Gamma \oplus \Theta$ which has the form specified in \eqref{eq41201} by
Proposition~\ref{prop441}.
Moreover, $\Gamma \oplus \Theta$ is isomorphic to the identity endomorphism of this image. This completes the proof.
\end{proof}

We would like to emphasize here the main difference between classical representation theory and $2$-representation theory. In classical representation theory, indecomposable projectives correspond to indecomposable idempotents. In $2$-representation theory, the idempotent associated to an indecomposable projective $2$-representation might not itself be indecomposable, but is the square of an indecomposable $1$-morphism (which itself might not be idempotent), see Example \ref{s7.2} for an example of this.

\section{$2$-Morita theory}\label{s3}

\subsection{$2$-progenerators}\label{s3.1}

Let $\cC$ be a finitary $2$-category. We denote by $\cC\text{-}\mathrm{proj}$ the full $2$-subcategory of $\cC\text{-}\mathrm{afmod}$ whose objects are projective $2$-representations.
A full $2$-subcategory $\cP$ of $\cC\text{-}\mathrm{proj}$ is called a {\em $2$-progenerator} provided that for any projective 
$2$-representation $\mathbf{P}$ of $\cC$ there is $\mathbf{P}'$ in the additive closure of $\cP$ and
strong transformations $\Psi:\mathbf{P}\to\mathbf{P}'$ and $\Phi:\mathbf{P}'\to\mathbf{P}$ such that
$\Phi\Psi$ is isomorphic to the identity on $\mathbf{P}$. For example, the $2$-subcategory 
$\cP_{\ccC,\mathbb{P}}$ of $\cC\text{-}\mathrm{afmod}$
whose objects are the principal $2$-representations of $\cC$ is a $2$-progenerator. 
Note that $\cP_{\ccC,\mathbb{P}}$ is biequivalent to $\cC^{\mathrm{op}}$ and $\mathrm{End}_{\ccC}(\cP_{\ccC,\mathbb{P}})$ is biequivalent to $\cC$.
Note also that  a full $2$-subcategory $\cP$ of $\cC\text{-}\mathrm{proj}$ is a progenerator if any 
principle $2$-representation is a retract of an object in the additive closure of $\cP$.

\subsection{The essential $2$-subcategory of a finitary $2$-category}\label{s2.8}

Let $\cC$ be a finitary $2$-category. Define a binary relation $\preceq$ on the set of equivalence classes
of objects of $\cC$ as follows: $\mathtt{i}\preceq\mathtt{j}$ if and only if there exist 
$\Phi\in\cC(\mathtt{j},\mathtt{i})$ and $\Psi\in\cC(\mathtt{i},\mathtt{j})$ such that 
$\Phi\Psi\cong\mathbbm{1}_{\mathtt{i}}$. Denote by ${}^{\dagger}{\cC}$ the full $2$-subcategory of $\cC$
given by a choice of one object $\mathtt{i}$ in each equivalence class which is a maximal 
element with respect to $\preceq$.
We will call ${}^{\dagger}{\cC}$ the {\em essential} $2$-subcategory of $\cC$.

\subsection{The main result}\label{s3.2}
The following is the main result of this paper. Recall that $\cA^{\mathrm{op}}$ is defined in Subsection~\ref{s2.1}.

\begin{theorem}[Morita theorem for finitary $2$-categories]\label{thmmain}
Let $\cA$ and $\cC$ be two finitary $2$-categories. Then the following assertions are equivalent. 
\begin{enumerate}[$($a$)$]
\item\label{thmmain.1} There is a $2$-progenerator $\cP$ for $\cC$ whose endomorphism $2$-category is
biequivalent to $\cA^{\mathrm{op}}$.
\item\label{thmmain.2} The $2$-categories $\cA\text{-}\mathrm{proj}$ and $\cC\text{-}\mathrm{proj}$ are biequivalent.
\item\label{thmmain.3} The $2$-categories $\cA\text{-}\mathrm{afmod}$ and $\cC\text{-}\mathrm{afmod}$ are biequivalent.
\end{enumerate}
\end{theorem}

\subsection{The implication \eqref{thmmain.3}$\Rightarrow$\eqref{thmmain.2}$\Rightarrow$\eqref{thmmain.1}}\label{s3.3}

Assume that $\cA\text{-}\mathrm{afmod}$ and $\cC\text{-}\mathrm{afmod}$ are biequivalent. Since
the notion of a projective $2$-representation is categorical, it follows that 
$\cA\text{-}\mathrm{proj}$ and $\cC\text{-}\mathrm{proj}$ are biequivalent. 
The image of $\cP_{\ccA,\mathbb{P}}$ under such biequivalence is a $2$-progenerator for $\cC$.

\subsection{The implication \eqref{thmmain.1}$\Rightarrow$\eqref{thmmain.2}}\label{s3.4}

Denote by $\underline{\cA}$ the endomorphism $2$-category of $\cP$. Then $\underline{\cA}^{\mathrm{op}}$ and
$\cA$ are biequivalent and hence $\cA\text{-}\mathrm{afmod}$ and 
$\underline{\cA}^{\mathrm{op}}\text{-}\mathrm{afmod}$ are biequivalent by Proposition~\ref{prop123}.
In particular, $\cA\text{-}\mathrm{proj}$ and 
$\underline{\cA}^{\mathrm{op}}\text{-}\mathrm{proj}$ are biequivalent by Subsection~\ref{s3.3}. It remains to show
that $\cC\text{-}\mathrm{proj}$ and $\underline{\cA}^{\mathrm{op}}\text{-}\mathrm{proj}$ are biequivalent.
We have the obvious $2$-functor 
\begin{displaymath}
\mathrm{Hom}_{\ccC}(\cP,{}_-):
\cC\text{-}\mathrm{proj}\to\underline{\cA}^{\mathrm{op}}\text{-}\mathrm{MOD}. 
\end{displaymath}
This $2$-functor sends objects of $\cP$ to principal, and hence projective,  $2$-rep\-re\-sen\-ta\-ti\-ons of 
$\underline{\cA}^{\mathrm{op}}$. Let $\mathbf{P}$ be a projective $2$-representation of $\cC$. By our definition of
a $2$-progenerator, there is $\mathbf{P}'$ in the additive closure of $\cP$ and
strong transformations $\Psi:\mathbf{P}\to\mathbf{P}'$ and $\Phi:\mathbf{P}'\to\mathbf{P}$ such that
$\Phi\Psi$ is isomorphic to the identity on $\mathbf{P}$. Applying $\mathrm{Hom}_{\ccC}(\cP,{}_-)$ we get
that $\mathrm{Hom}_{\ccC}(\cP,\mathbf{P}')$ is a direct sum of principal $2$-representations of $\underline{\cA}^{\mathrm{op}}$, and $\mathrm{Hom}_{\ccC}(\cP,\mathbf{P})$ comes together with
two maps $\mathrm{Hom}_{\ccC}(\cP,\Phi)$ and $\mathrm{Hom}_{\ccC}(\cP,\Psi)$ such that 
$\mathrm{Hom}_{\ccC}(\cP,\Phi\Psi)$ is isomorphic to the identity on $\mathrm{Hom}_{\ccC}(\cP,\mathbf{P})$.
Since $\mathrm{Hom}_{\ccC}(\cP,\mathbf{P}')$ is a projective $2$-representation of 
$\underline{\cA}^{\mathrm{op}}$, so is $\mathrm{Hom}_{\ccC}(\cP,\mathbf{P})$. Therefore
\begin{equation}\label{eq505}
\mathrm{Hom}_{\ccC}(\cP,{}_-):
\cC\text{-}\mathrm{proj}\to\underline{\cA}^{\mathrm{op}}\text{-}\mathrm{proj}.
\end{equation}

Let $\mathtt{I}:=\{\mathtt{i}_1,\dots,\mathtt{i}_k\}$ be a cross-section of equivalence classes of objects in $\cC$.
Let $\mathbb{P}:=\mathbb{P}_{\mathtt{i}_1}\oplus\cdots\oplus \mathbb{P}_{\mathtt{i}_k}$ and
$\mathbf{P}$ a $2$-representation in the additive closure of $\cP$ such that there exist strong transformations 
$\Psi:\mathbb{P}\to\mathbf{P}$ and $\Phi:\mathbf{P}\to\mathbb{P}$ with the property that
$\Phi\Psi$ is isomorphic to the identity on $\mathbb{P}$. Let $\mathbf{M}$ and $\mathbf{N}$ be two 
$2$-representations of $\cC$, let further $\Lambda,\Lambda':\mathbf{M}\to \mathbf{N}$ be strong transformations
and $\eta:\Lambda\to\Lambda'$ a modification. Then the Yoneda Lemma says that $\mathrm{Hom}_{\ccC}(\mathbb{P},{}_-)$ 
maps
\begin{displaymath}
\mathbf{M}\mapsto \bigoplus_{\mathtt{i}\in\mathtt{I}}  \mathbf{M}(\mathtt{i}),\quad
\mathbf{N}\mapsto \bigoplus_{\mathtt{i}\in\mathtt{I}}  \mathbf{N}(\mathtt{i}),\quad
\Lambda\mapsto \bigoplus_{\mathtt{i}\in\mathtt{I}} \Lambda_{\mathtt{i}},\quad
\Lambda'\mapsto \bigoplus_{\mathtt{i}\in\mathtt{I}} \Lambda'_{\mathtt{i}},\quad
\eta\mapsto \bigoplus_{\mathtt{i}\in\mathtt{I}} \eta_{\mathtt{i}}.
\end{displaymath}
Since $\Phi\Psi$ is isomorphic to the identity on $\mathbb{P}$, it follows that composition with
$\Phi$ maps $\Lambda$, $\Lambda'$ or $\eta$ to zero if and only if $\Lambda=0$, $\Lambda'=0$ or $\eta=0$,
respectively. This implies that the $2$-functor $\mathrm{Hom}_{\ccC}(\cP,{}_-)$ from \eqref{eq505} is locally
faithful.

Let $\mathbf{Q}$ denote the multiplicity-free direct sum of all objects in $\cP$ up to equivalence
(in particular, this sum is finite). By construction, $\mathbb{P}$ is a $2$-progenerator for $\cC$. Hence
there is a direct sum $\mathbf{R}$ of principal $2$-representations of $\cC$ and strong transformations
$\Psi':\mathbf{Q}\to\mathbf{R}$ and $\Phi':\mathbf{R}\to\mathbf{Q}$ such that $\Phi'\Psi'$ is isomorphic to the 
identity on $\mathbf{Q}$. Since $\cC^{\mathrm{op}}$ is biequivalent to $\cP_{\ccC,\mathbb{P}}$,
the $2$-functor 
\begin{displaymath}
\mathrm{Hom}_{\ccC}(\cP_{\ccC,\mathbb{P}},{}_-): \cC\text{-}\mathrm{proj}\to
\mathrm{End}(\cP_{\ccC,\mathbb{P}})^{\mathrm{op}}\text{-}\mathrm{afmod}
\end{displaymath}
is locally full and dense by Proposition~\ref{prop123}.
As the additive closures of $\cP_{\ccC,\mathbb{P}}$ and $\mathbb{P}$ coincide, it follows that
the $2$-functor 
\begin{displaymath}
\mathrm{Hom}_{\ccC}(\mathbb{P},{}_-): \cC\text{-}\mathrm{proj}\to
\mathrm{End}(\mathbb{P})^{\mathrm{op}}\text{-}\mathrm{afmod}
\end{displaymath}
is locally full and dense. Since $\mathbf{R}$ belongs to the
additive closure of $\mathbb{P}$, we get  that the
$2$-functor 
\begin{displaymath}
\mathrm{Hom}_{\ccC}(\mathbf{R},{}_-): \cC\text{-}\mathrm{proj}\to
\mathrm{End}(\mathbf{R})^{\mathrm{op}}\text{-}\mathrm{afmod}
\end{displaymath}
is locally full and dense. Finally, since $\Phi'\Psi'$ is isomorphic to 
the identity on $\mathbf{Q}$, composition with $\Psi'$ gives that the
$2$-functor 
\begin{displaymath}
\mathrm{Hom}_{\ccC}(\mathbf{Q},{}_-): \cC\text{-}\mathrm{proj}\to
\mathrm{End}(\mathbf{Q})^{\mathrm{op}}\text{-}\mathrm{afmod}
\end{displaymath}
is locally full and dense. The latter implies that the $2$-functor $\mathrm{Hom}_{\ccC}(\cP,{}_-)$ from \eqref{eq505} 
is locally full and dense.

It remains to show that the $2$-functor $\mathrm{Hom}_{\ccC}(\cP,{}_-)$ from \eqref{eq505} is surjective 
on equivalence classes of objects.
Let now $\mathbf{P}$ be a projective $2$-representation of $\underline{\cA}^{\mathrm{op}}$. Without 
loss of generality we may assume that $\underline{\cA}^{\mathrm{op}}$ is cancellative. By 
Theorem~\ref{thm412}, $\mathbf{P}$ is equivalent to a direct sum of $2$-rep\-re\-sen\-tations of the 
form $\mathbf{P}^{\underline{\ccA}^{\mathrm{op}}}_{\mathtt{i},\mathrm{E}}$ 
for some $\mathtt{i}\in\underline{\cA}^{\mathrm{op}}$ and 
$\mathrm{E}\in\underline{\cA}^{\mathrm{op}}(\mathtt{i},\mathtt{i})$ such that $\mathrm{E}^2\cong\mathrm{E}$. 
From the definition of $\underline{\cA}^{\mathrm{op}}$ we have that any such $\mathrm{E}$ comes from an
idempotent endomorphism of an object $\mathbf{Q}\in\cP$. Since $\underline{\cA}^{\mathrm{op}}$ is finitary
(in particular, $\mathbbm{1}_{\mathtt{i}}$ is indecomposable),
$\mathbf{Q}$ has the form $\mathbf{P}^{\ccC}_{\mathtt{j},\mathrm{F}}$ for some
$\mathtt{j}\in\cC$ and $\mathrm{F}^2\cong\mathrm{F}\in\cC(\mathtt{j},\mathtt{j})$.
From Corollary~\ref{cor437}\eqref{cor437.3} it follows that $\mathrm{E}$ comes from some
$\mathrm{H}^2\cong\mathrm{H}\in \mathrm{F}\circ\cC(\mathtt{j},\mathtt{j})\circ\mathrm{F}$.
Consequently, $\mathrm{Hom}_{\ccC}(\cP,\mathbf{P}^{\ccC}_{\mathtt{j},\mathrm{H}})\cong 
\mathbf{P}^{\underline{\ccA}^{\mathrm{op}}}_{\mathtt{i},\mathrm{E}}$ and hence $\mathrm{Hom}_{\ccC}(\cP,{}_-)$
is surjective on equivalence classes of objects.

\subsection{The implication \eqref{thmmain.2}$\Rightarrow$\eqref{thmmain.3}}\label{s3.6}

We start with the following observation (recall that the definition of the essential $2$-subcategories
${}^{\dagger}\hspace{-1mm}\cA$ and ${}^{\dagger}\cC$ is in Section~\ref{s2.8}).

\begin{lemma}\label{lem531}
Under assumption \eqref{thmmain.2} the $2$-categories ${}^{\dagger}\hspace{-1mm}\cA$ and 
${}^{\dagger}\cC$ are biequivalent.
\end{lemma}

\begin{proof}
The order $\preceq$ extends in the obvious way to all $\mathbf{P}_{\mathtt{i},\mathrm{E}}$ in 
$\cA\text{-}\mathrm{proj}$ (and similarly for  $\cC\text{-}\mathrm{proj}$). Note that maximal elements
with respect to $\preceq$ will be principal. Any biequivalence between $\cA\text{-}\mathrm{proj}$ 
and $\cC\text{-}\mathrm{proj}$ maps maximal elements with respect to $\preceq$ (for $\cA$) to maximal elements
with respect to $\preceq$ (for $\cC$) and hence induces a biequivalence between
${}^{\dagger}\hspace{-1mm}\cA$ and ${}^{\dagger}\cC$.
\end{proof}

Now from Proposition~\ref{prop123} we have that ${}^{\dagger}\hspace{-1mm}\cA\text{-}\mathrm{afmod}$ 
and ${}^{\dagger}\cC\text{-}\mathrm{afmod}$ are biequivalent. The proof of the implication
\eqref{thmmain.2}$\Rightarrow$\eqref{thmmain.3} is now completed by the following:

\begin{proposition}\label{prop555}
The $2$-categories $\cC\text{-}\mathrm{afmod}$ and ${}^{\dagger}\cC\text{-}\mathrm{afmod}$ 
are biequivalent. 
\end{proposition}

\begin{proof}
We have the restriction functor 
\begin{displaymath}
\mathrm{Res}:\cC\text{-}\mathrm{afmod}\to{}^{\dagger}\cC\text{-}\mathrm{afmod}. 
\end{displaymath}
Let $\cP$ be a full subcategory of ${}^{\dagger}\cC\text{-}\mathrm{proj}$ consisting of
restrictions to ${}^{\dagger}\cC$ of all principal $2$-representations of $\cC$.
As in the previous subsection we have the obvious $2$-functor
\begin{displaymath}
\mathrm{Hom}_{{}^{\dagger}\ccC}(\cP,{}_-):
{}^{\dagger}\cC\text{-}\mathrm{afmod}\to\cC\text{-}\mathrm{afmod}. 
\end{displaymath}
Moreover, it is easy to check that $\mathrm{Res}\circ\mathrm{Hom}_{{}^{\dagger}\ccC}(\cP,{}_-)$ is isomorphic to the
identity on ${}^{\dagger}\cC\text{-}\mathrm{afmod}$. In particular, 
$\mathrm{Hom}_{{}^{\dagger}\ccC}(\cP,{}_-)$ maps non-equivalent
$2$-representations of ${}^{\dagger}\cC$ to non-equivalent $2$-representations of $\cC$, moreover, it is 
locally faithful and locally injective on isomorphism classes of objects.

Let $\mathbf{M}$ and $\mathbf{N}$ be  $2$-representations of $\cC$. Let $\Lambda,\Lambda':\mathbf{M}\to \mathbf{N}$ 
be strong transformations and $\alpha:\Lambda\to\Lambda'$ a modification. 
Let $\mathtt{j}$ be an object of ${}^{\dagger}\cC$ and 
$\mathtt{i}$ be an object of $\cC\setminus{}^{\dagger}\cC$ such that $\mathtt{i}\preceq\mathtt{j}$. 
Following the notation of Subsection~\ref{s2.8}, we have the commutative (up to natural isomorphism) diagram
\begin{displaymath}
\xymatrix{ 
\mathbf{M}(\mathtt{j})\ar@/^/[rr]^{\mathbf{M}(\Phi)}
\dtwocell^{\Lambda_{\mathtt{j}}}_{\Lambda'_{\mathtt{j}}}{\alpha_{\mathtt{j}}}
&&\ar@/^/[ll]^{\mathbf{M}(\Psi)}\mathbf{M}(\mathtt{i})
\dtwocell^{\Lambda_{\mathtt{i}}}_{\Lambda'_{\mathtt{i}}}{\alpha_{\mathtt{i}}}\\
\mathbf{N}(\mathtt{j})\ar@/^/[rr]^{\mathbf{N}(\Phi)}&&\ar@/^/[ll]^{\mathbf{N}(\Psi)}\mathbf{N}(\mathtt{i})
} 
\end{displaymath}
from which we have isomorphisms
\begin{displaymath}
\xi_1:\Lambda_{\mathtt{i}}\to\mathbf{N}(\Phi)\Lambda_{\mathtt{j}}\mathbf{M}(\Psi)\quad\text{ and }\quad
\xi_2:\Lambda'_{\mathtt{i}}\to \mathbf{N}(\Phi)\Lambda'_{\mathtt{j}}\mathbf{M}(\Psi) .
\end{displaymath}
Furthermore,  we also have
\begin{displaymath}
\alpha_{\mathtt{i}}=\xi_2^{-1}\circ_1(\mathrm{id}_{\mathbf{N}(\Phi)}\circ_0 \alpha_{\mathtt{j}}\circ_0 
\mathrm{id}_{\mathbf{M}(\Psi)})\circ_1\xi_1.
\end{displaymath}
This proves that $\mathrm{Res}$ maps non-equivalent
$2$-representations of $\cC$ to non-equivalent $2$-representations of ${}^{\dagger}\cC$, moreover, it also
proves that $\mathrm{Res}$ is locally faithful and locally injective on isomorphism classes of objects.

Combined with the previous paragraph we have that 
\begin{itemize}
\item $\mathrm{Res}$ and  $\mathrm{Hom}_{{}^{\dagger}\ccC}(\cP,{}_-)$ induce mutually inverse 
bijections between equivalence classes of objects,
\item locally they induce mutually inverse bijections between isomorphism classes of $1$-morphisms,
\item locally they induce injections in both directions on the level of $2$-morphisms.
\end{itemize}
Since $2$-morphism spaces are finite dimensional, we obtain that the $2$-functors $\mathrm{Res}$ and 
$\mathrm{Hom}_{{}^{\dagger}\ccC}(\cP,{}_-)$ are mutually inverse biequivalences.
\end{proof}

\section{Examples}\label{s7}

\subsection{Morita equivalent but not biequivalent finitary $2$-categories}\label{s7.1}

Denote by $A$ the path algebra of the quiver $\xymatrix{1\ar[r]&2}$ over $\Bbbk$ and let
$\mathcal{C}$ be a small category equivalent to $A\text{-}\mathrm{proj}$. Denote by $\mathrm{F}$ the
endofunctor of $\mathcal{C}$ given by tensoring with $Ae_{2}\otimes_{\Bbbk}e_{2}A$. Clearly,
$\mathrm{F}^2\cong \mathrm{F}$. Consider the $2$-category $\cC$ defined as follows: 
\begin{itemize}
\item $\cC$ has one object $\mathtt{i}$ (which we identify with $\mathcal{C}$);
\item $1$-morphisms in $\cC$ are all functors which are isomorphic to a direct sum of copies of
$\mathrm{F}$ and the identity functor $\mathbbm{1}_{\mathtt{i}}=\mathrm{Id}_{\mathcal{C}}$;
\item $2$-morphisms in $\cC$ are all natural transformations of functors.
\end{itemize}
Let $\underline{\cA}$ denote the full $2$-subcategory of $\cC\text{-}\mathrm{proj}$ with objects
$\mathbb{P}_{\mathtt{i}}$ and $\mathbf{P}_{\mathtt{i},\mathrm{F}}$. Set $\cA:=\underline{\cA}^{\mathrm{op}}$.
Then $\cA$ and $\cC$ are Morita equivalent by Theorem~\ref{thmmain}. On the other hand,
$\cA$ has two objects which are not equivalent (since the functors representing the actions of 
$\mathbbm{1}_{\mathtt{i}}$ and $\mathrm{F}$ on $\mathbf{P}_{\mathtt{i},\mathrm{F}}$ are isomorphic while 
 the functors representing the actions of $\mathbbm{1}_{\mathtt{i}}$ 
and $\mathrm{F}$ on $\mathbb{P}_{\mathtt{i}}$ are not isomorphic). At the same time $\cC$ has only one object.
Hence $\cA$ and $\cC$ are not biequivalent.

\subsection{Indecomposable non-idempotent $1$-morphisms which square to idempotent $1$-morphisms}\label{s7.2}

Take $A=\Bbbk\oplus\Bbbk\oplus\Bbbk$ over $\Bbbk$, let $\mathcal{D}$ be a small category equivalent to
$\Bbbk\text{-}\mathrm{mod}$ and $\mathcal{C}:=\mathcal{D}\oplus\mathcal{D}\oplus\mathcal{D}$ (which
is equivalent to  $A\text{-}\mathrm{proj}=A\text{-}\mathrm{mod}$). Denote by $\mathrm{E}$ the identity 
functor on $\mathcal{D}$ and let $\mathrm{F}$ and $\mathrm{K}$ be the endofunctors of $\mathcal{C}$ 
given by the matrices
\begin{displaymath}
\left(\begin{array}{ccc}0&\mathrm{E}&0\\0&\mathrm{E}&\mathrm{E}\\0&0&0\end{array}\right)\quad\text{and}\quad 
\left(\begin{array}{ccc}0&0&\mathrm{E}\\0&0&0\\0&0&0\end{array}\right),
\end{displaymath}
respectively. We have $\mathrm{K}^2=\mathrm{K}\mathrm{F}=\mathrm{F}\mathrm{K}=0$ and
$\mathrm{F}^2\cong \mathrm{F}\oplus\mathrm{K}$, from which it follows that $(\mathrm{F}\oplus\mathrm{K})^2\cong
\mathrm{F}\oplus\mathrm{K}$. Consider the $2$-category $\cC$ defined as follows: 
\begin{itemize}
\item $\cC$ has one object $\mathtt{i}$ (which we identify with $\mathcal{C}$);
\item $1$-morphisms in $\cC$ are all functors which are isomorphic to a direct sum of copies of
$\mathrm{F}$, $\mathrm{K}$ and the identity functor $\mathbbm{1}_{\mathtt{i}}=\mathrm{Id}_{\mathcal{C}}$;
\item $2$-morphisms in $\cC$ are given by scalar multiples of the identity natural transformations on 
$\mathrm{F}$, $\mathrm{K}$ and $\mathbbm{1}_{\mathtt{i}}$, extended additively to their direct sums.
\end{itemize}
The $1$-morphism $\mathrm{F}$ in $\cC$ is an indecomposable non-idempotent $1$-morphism which squares to 
an idempotent (but decomposable) $1$-morphism. 

\subsection{Morita equivalence classes for $2$-categories of projective functors for
finite dimensional algebras}\label{s7.3}

Let $A$ be a finite dimensional $\Bbbk$-algebra. Assume that $A\cong A_1\oplus A_2\oplus\dots\oplus A_k$
with $A_1,A_2,\dots,A_k$ connected (that is indecomposable as algebras).
Denote by $\cC_{A}$ the $2$-category defined as follows (compare  \cite[7.3]{MM}):
\begin{itemize}
\item objects are $\mathtt{1},\dots,\mathtt{k}$ where we identify $\mathtt{i}$ with some small category equivalent
to $A_i\text{-}\mathrm{proj}$;
\item $1$-morphisms are all additive functors from $A_i\text{-}\mathrm{proj}$ to $A_j\text{-}\mathrm{proj}$ 
isomorphic to direct sums of functors realized as tensoring with either projective $A_j\text{-}A_i$-bimodules
or, additionally, with the bimodule $A_i$ in case $i=j$;
\item $2$-morphisms are natural transformations of functors.
\end{itemize}
Note that, up to biequivalence, $\cC_{A}$ does not depend on the choice of small categories equivalent
to $A_i\text{-}\mathrm{proj}$ for $i=1,2,\dots,k$. The $2$-category $\cC_{A}$ is of particular interest, as in \cite[Theorem 13]{MM3} it was shown that fiat $2$-categories, which are ``simple'' in a certain sense, are constructed from these.

For $B\cong B_1\oplus B_2\oplus\dots\oplus B_m$ with  $B_1,B_2,\dots,B_m$ connected, write
$A\sim B$ provided that the following two conditions are satisfied:
\begin{itemize}
\item $m=k+1$, $A_1\cong B_1,A_2\cong B_2$,\dots, $A_k\cong B_k$ and 
$B_{m}\cong\Bbbk$;
\item there are idempotents $e,e'\in A$ such that $\dim eAe'=1$. 
\end{itemize}
Denote by $\approx$ the minimal equivalence relation (on the class of all finite dimensional $\Bbbk$-algebra)
containing both $\sim$ and the classical Morita equivalence relation for finite dimensional algebras.

\begin{theorem}\label{thm97}
Let $A$ and $B$ be two  finite dimensional $\Bbbk$-algebras. Then $\cC_{A}$ and $\cC_{B}$ are Morita equivalent
if and only if $A\approx B$.
\end{theorem}

\begin{proof}
We first note that the $2$-category $\cC_{A}$ is, clearly, independent of the choice
of $A$ within its Morita equivalence class. Hence, to prove sufficiency 
it is enough to show that $A\sim B$ implies Morita equivalence of $\cC_{A}$ and $\cC_{B}$.
Let $\mathbb{P}^{\ccC_{B}}_{\mathtt{m}}$ be the principal $2$-representation of $B$ associated to $B_m$.
We identify $A$ with the subalgebra $B_1\oplus\dots\oplus B_k$ of $B$.
Let $x$ denote the identity in $B_m$. Then $Bx\otimes_{\Bbbk}eA$ and $Ae'\otimes_{\Bbbk}xB$ are $1$-morphisms in 
$\cC_B$. As $\dim eAe'=1$, we have
\begin{displaymath}
Bx\otimes_{\Bbbk}eA\otimes_B Ae'\otimes_{\Bbbk}xB \cong Bx\otimes_{\Bbbk}xB^{\oplus \dim eAe'}\cong B_m. 
\end{displaymath}
This implies that $\mathbb{P}^{\ccC_{B}}_{\mathtt{m}}$ is a retract of 
$\mathbb{P}^{\ccC_{B}}_{\mathtt{1}}\oplus\dots\oplus\mathbb{P}^{\ccC_{B}}_{\mathtt{k}}$. By our choice of $A$
we have that the endomorphism $2$-category of $\mathbb{P}^{\ccC_{B}}_{\mathtt{1}},
\mathbb{P}^{\ccC_{B}}_{\mathtt{2}},\dots,\mathbb{P}^{\ccC_{B}}_{\mathtt{k}}$ 
is biequivalent to the endomorphism $2$-category of
$\mathbb{P}^{\ccC_{A}}_{\mathtt{1}},\mathbb{P}^{\ccC_{A}}_{\mathtt{2}},\dots,\mathbb{P}^{\ccC_{A}}_{\mathtt{k}}$. 
The claim follows.
 
To prove necessity, let us analyze idempotent $1$-morphisms in $\cC_A$. Let $e,e'$ be primitive idempotents
in $A$ and $Ae'\otimes_{\Bbbk}eA$ the corresponding projective bimodule. We have
\begin{equation}\label{eq976}
Ae'\otimes_{\Bbbk}eA\otimes_A Ae'\otimes_{\Bbbk}eA\cong Ae'\otimes_{\Bbbk}eA ^{\oplus \dim eAe'}
\end{equation}
and hence
\begin{multline*}
(Ae'\otimes_{\Bbbk}eA\otimes_A Ae'\otimes_{\Bbbk}eA)\otimes_A 
(Ae'\otimes_{\Bbbk}eA\otimes_A Ae'\otimes_{\Bbbk}eA)\\\cong 
(Ae'\otimes_{\Bbbk}eA\otimes_A Ae'\otimes_{\Bbbk}eA) ^{\oplus (\dim eAe')^2}.
\end{multline*}
This implies that $Ae'\otimes_{\Bbbk}eA\otimes_A Ae'\otimes_{\Bbbk}eA$ is idempotent if and only if
$\dim eAe'=1$. Note that, by \eqref{eq976}, the latter is equivalent to $Ae'\otimes_{\Bbbk}eA$ being idempotent.
By Theorem~\ref{thmmain}, the Morita equivalence class of a finitary $2$-category is obtained by adding or
removing retracts of indecomposable principal $2$-representations. Let $\mathbf{P}$ be a retract corresponding to an idempotent $Ae'\otimes_{\Bbbk}eA$. Since $Ae'\otimes_{\Bbbk}eA$ is idempotent, it factors through 
$\Bbbk\text{-}\mathrm{proj}$.  This implies that 
the endomorphism $2$-category of $\mathbb{P}_{\mathtt{1}},\mathbb{P}_{\mathtt{2}},\dots, \mathbb{P}_{\mathtt{k}},
\mathbf{P}$ is biequivalent to $\cC_B$ where $B=A\oplus \Bbbk$. This completes the proof.
\end{proof}

\subsection{$2$-categories of Soergel bimodules}\label{s7.4}

Let $(W,S)$ be a finite Coxeter system and $\mathbf{C}_W$ the corresponding coinvariant algebra (see \cite{Wi}
for details). Let $\cS_{(W,S)}$ be a $2$-category of Soergel $\mathbf{C}_W\text{-}\mathbf{C}_W$-bimodules for 
$(W,S)$ as defined in \cite[Subsection~7.1]{MM} or \cite[Example~3]{MM2} (in those papers $W$ is assumed to be a 
Weyl group, however, our more general assumption works just fine, see \cite{EW,Wi}). 
The $2$-category $\cS_{(W,S)}$ is usually described using its defining $2$-representation:
\begin{itemize}
\item $\cS_{(W,S)}$ has one object which is identified with some small category $\mathcal{A}$ equivalent to
$\mathbf{C}_W\text{-}\mathrm{mod}$;
\item $1$-morphisms of $\cS_{(W,S)}$ are endofunctors of $\mathcal{A}$ isomorphic to 
direct sums of endofunctors given by tensoring with Soergel bimodules;
\item $2$-morphisms of $\cS_{(W,S)}$ are natural transformations of functors.
\end{itemize}
Clearly, up to biequivalence, $\cS_{(W,S)}$ does not depend on the choice of $\mathcal{A}$.

\begin{proposition}\label{prop983}
Let $(W,S)$ and $(W',S')$ be finite Coxeter systems. Then the $2$-categories 
$\cS_{(W,S)}$ and $\cS_{(W',S')}$ are Morita equivalent if and only if $(W,S)$ and $(W',S')$ are isomorphic.
\end{proposition}

\begin{proof}
The ``if'' part is obvious. The ``only if'' part follows from Theorem~\ref{thmmain} and the following two
observations.

{\em Observation 1. The $2$-categories $\cS_{(W,S)}$ and $\cS_{(W',S')}$ are biequivalent if and only if
$(W,S)$ and $(W',S')$ are isomorphic.} The ``if'' part is again obvious. To prove the ``only if'' part, 
note that any biequivalence between 
$\cS_{(W,S)}$ and $\cS_{(W',S')}$ induces an isomorphism between the $2$-endomorphism algebras of the
identity $1$-morphisms in $\cS_{(W,S)}$ and $\cS_{(W',S')}$. By definition, these endomorphism algebras
are isomorphic to $\mathbf{C}_W$ and $\mathbf{C}_{W'}$, respectively. Finally, it is easy to check that
$\mathbf{C}_W\cong\mathbf{C}_{W'}$ if and only if $(W,S)\cong(W',S')$.

{\em Observation 2. The only weakly idempotent $1$-morphism in $\cS_{(W,S)}$ is the identity $1$-morphism.} 
Tensoring with a Soergel bimodule is a non-zero and exact endofunctor of $\mathcal{A}$ and the latter category
has only one simple object up to isomorphism (call it $L$). Let $\mathrm{G}$ be an indecomposable
$1$-morphism in $\cS_{(W,S)}$ such that $\mathrm{F}:=\mathrm{G}\circ\mathrm{G}\cong \mathrm{G}\oplus\mathrm{Q}$
is weakly idempotent. Let $m$ be the length of $\mathrm{F}(L)\cong\mathrm{F}^2(L)$. Then $m=m^2$ and hence
$m=1$ (since $\mathrm{G}$ is nonzero). It follows that $\mathrm{Q}=0$ and thus $\mathrm{F}=\mathrm{G}$.
In particular, $\mathrm{F}$ is indecomposable and hence corresponds to some element $w\in W$ and, moreover,
$\mathrm{F}(L)\cong L$. In the natural graded picture $\mathrm{F}(L)$ is a graded self-dual vector space
with non-zero components in degrees $\pm$ the length of $w$. Therefore $\mathrm{F}(L)\cong L$ implies
that the length of $w$ equals zero and hence $w$ must coincide with the identity element. Therefore 
$\mathrm{F}$ is isomorphic to the identity $1$-morphism. This completes the proof.
\end{proof}

%\vspace{1cm}

\noindent
Volodymyr Mazorchuk, Department of Mathematics, Uppsala University,
Box 480, 751 06, Uppsala, SWEDEN, {\tt mazor\symbol{64}math.uu.se};
http://www.math.uu.se/$\tilde{\hspace{1mm}}$mazor/.

%\vspace{0.1cm}
\noindent
Vanessa Miemietz, School of Mathematics, University of East Anglia,
Norwich, UK, NR4 7TJ, {\tt v.miemietz\symbol{64}uea.ac.uk};
http://www.uea.ac.uk/$\tilde{\hspace{1mm}}$byr09xgu/.


\begin{thebibliography}{999999}
%\bibitem[Ag]{Ag} T.~Agerholm; Simple $2$-representations and classification of categorifications. PhD Thesis,
%Aarhus University, 2011.
%\bibitem[AM]{AM} T.~Agerholm, V.~Mazorchuk; On selfadjoint functors 
%satisfying polynomial relations, Preprint arXiv:1004.0094,
%to appear in J. Algebra.
%\bibitem[AS]{AS} H.~Andersen, C.~Stroppel; Twisting functors on $\mathcal{O}$. Represent. Theory 
%{\bf 7} (2003), 681--699.
\bibitem[Ar]{Ar} S.~Ariki; On the decomposition numbers of the Hecke algebra of $G(m,1,n)$. 
J. Math. Kyoto Univ. {\bf 36} (1996), no. 4, 789--808.
%\bibitem[Au]{Au} M.~Auslander; Representation theory of Artin algebras. 
%I, II.  Comm. Algebra  {\bf 1}  (1974), 177--268; ibid. {\bf 1} 
%(1974), 269--310. 
%\bibitem[BFK]{BFK} J.~Bernstein, I.~Frenkel, M.~Khovanov; A categorification of the Temperley-Lieb algebra and Schur 
%quotients of $U(\mathfrak{sl}_2)$ via projective and Zuckerman functors. Selecta Math. (N.S.) {\bf 5} (1999), no. 
%2, 199--241. 
%\bibitem[BGG]{BGG} J.~Bernstein, I.~Gelfand, S.~Gelfand; A certain category of ${\mathfrak g}$-modules. 
%Funkcional. Anal. i Prilo{\v z}en. {\bf 10} (1976), no. 2, 1--8. 
%\bibitem[BG]{BG} J.~Bernstein, S.~Gelfand; Tensor products of 
%finite- and infinite-dimensional representations of semisimple Lie 
%algebras.  Compositio Math.  {\bf 41}  (1980), no. 2, 245--285. 
%\bibitem[Be]{Be} R.~Bezrukavnikov; On tensor categories attached to 
%cells in affine Weyl groups.  Representation theory of algebraic 
%groups and quantum groups,  69--90, Adv. Stud. Pure Math., {\bf 40}, 
%Math. Soc. Japan, Tokyo, 2004. 
\bibitem[BD]{BD} F.~Borceux, D.~Dejean; Cauchy completion in category theory. Cahiers 
Topologie G{\'e}om. Diff{\'e}rentielle Cat{\'e}g. {\bf 27} (1986), no. 2, 133--146.
\bibitem[BV]{BV} G.~B{\"o}hm, J.~Vercruysse; Morita theory for comodules over corings. 
Comm. Algebra {\bf 37} (2009), no. 9, 3207--3247.
\bibitem[CR]{CR} J.~Chuang, R.~Rouquier; Derived equivalences for 
symmetric groups and $\mathfrak{sl}_2$-ca\-te\-go\-ri\-fi\-ca\-ti\-on. 
Ann. of Math.  (2) {\bf 167} (2008), no. 1, 245--298. 
%\bibitem[Cr]{Cr} L.~Crane; Clock and category: is quantum gravity 
%algebraic?  J. Math. Phys. {\bf 36}  (1995),  no. 11, 6180--6193. 
%\bibitem[CF]{CF} L.~Crane, I.~Frenkel; Four-dimensional topological 
%quantum field theory, Hopf categories, and the canonical bases. 
%Topology and physics. J. Math. Phys. {\bf 35} (1994), no. 10, 5136--5154.
%\bibitem[Du]{Du} J.~Du; Kazhdan-Lusztig bases and isomorphism theorems for $q$-Schur algebras. 
%Kazhdan-Lusztig theory and related topics (Chicago, IL, 1989), 121--140,
%Contemp. Math., {\bf 139}, Amer. Math. Soc., Providence, RI, 1992.
\bibitem[EW]{EW} B.~Elias, G.~Williamson; The Hodge theory of Soergel bimodules.
Preprint arXiv:1212.0791.
%\bibitem[EGNO]{EGNO} P.~Etingof, S.~Gelaki, D.~Nikshych, A
%V.~Ostrik; Tensor categories, manuscript, available from
%http://www-math.mit.edu/$\sim$etingof/tenscat.pdf
%\bibitem[EO]{EO} P.~Etingof, V.~Ostrik; Finite tensor categories.  
%Mosc. Math. J.  {\bf 4}  (2004),  no. 3, 627--654, 782--783.
\bibitem[Fl]{Fl} P.~Flor; On groups of non-negative matrices. Compositio Math. {\bf 21} (1969), 376--382.
%\bibitem[FKS]{FKS} I.~Frenkel, M.~Khovanov, C.~Stroppel; A categorification of finite-dimensional irreducible 
%representations of quantum $\mathfrak{sl}_2$ and their tensor products. Selecta Math. (N.S.) {\bf 12} (2006), 
%no. 3-4, 379--431.
%\bibitem[Fr]{Fr} P.~Freyd; Representations in abelian categories. 
%in: Proc. Conf. Categorical Algebra (1966), 95--120.
%\bibitem[GL]{GL} J.~Graham, G.~Lehrer; Cellular 
%algebras. Invent. Math. {\bf 123} (1996), no. 1, 1--34.
%\bibitem[Gr]{Gr} J.~Green; On the structure of semigroups. Ann. of Math. (2) {\bf 54}, (1951). 163--172. 
\bibitem[Gr]{Gr} I.~Grojnowski; Affine $\mathfrak{sl}_p$ controls the representation theory of the symmetric group 
and related Hecke algebras. Preprint arXiv:math/9907129.
%\bibitem[Hi]{Hi} H.~Hiller; Geometry of Coxeter groups. Research Notes in Mathematics, {\bf 54}. Pitman 
%(Advanced Publishing Program), Boston, Mass.--London, 1982.
%\bibitem[Hu]{Hu} J.~Humphreys; Representations of semisimple Lie 
%algebras in the BGG category $\mathcal{O}$. Graduate Studies in 
%Mathematics, {\bf 94}. American Mathematical Society, Providence, RI, 2008.
%\bibitem[KaLu]{KaLu} D.~Kazhdan, G.~Lusztig; 
%Representations of Coxeter groups and Hecke algebras. 
%Invent. Math. {\bf 53} (1979), no. 2, 165--184.
%\bibitem[Ke]{Ke} B.~Keller; Deriving DG categories. Ann. Sci. {\'E}cole Norm. Sup. (4) {\bf 27} (1994), 
%no. 1, 63--102. 
\bibitem[Ke]{Ke} B.~Keller; On differential graded categories. International Congress of Mathematicians. 
Vol. II, 151--190, Eur. Math. Soc., Z{\"u}rich, 2006.
%\bibitem[Kh]{Kh} O.~Khomenko; Categories with projective functors.  
%Proc. London Math. Soc. (3)  {\bf 90}  (2005),  no. 3, 711--737.
%\bibitem[KM]{KM} O.~Khomenko, V.~Mazorchuk; On Arkhipov's and 
%Enright's functors.  Math. Z. {\bf 249}  (2005),  no. 2, 357--386.
\bibitem[Kh]{Kh} M.~Khovanov; A categorification of the Jones polynomial. Duke Math. J. {\bf 101} 
(2000), no. 3, 359--426.
%\bibitem[Kv]{Kv} M.~Khovanov; A functor-valued invariant of tangles.  
%Algebr. Geom. Topol. {\bf 2}  (2002), 665--741 (electronic).
%\bibitem[KhLa]{KL} M.~Khovanov, A.~Lauda; A categorification of A
%quantum $\mathfrak{sl}_n$. Quantum Topol. {\bf 1} (2010), 1--92.
%\bibitem[KMS]{KMS} M.~Khovanov, V.~Mazorchuk, C.~Stroppel; 
%A categorification of integral Specht modules.  Proc. Amer.
%{\bf 136}  (2008),  no. 4, 1163--1169.
\bibitem[Kn]{Kn} U.~Knauer; Projectivity of acts and Morita equivalence of monoids. Semigroup Forum {\bf 3} 
(1971/1972) no. 4, 359--370.
%\bibitem[Kn]{Kn} D.~Knuth; Permutations, matrices, and generalized Young tableaux. Pacific J. Math. {\bf 234} 
%(1970), 709--727. 
%\bibitem[KM]{KM} G.~Kudryavtseva, V.~Mazorchuk; On multisemigroups. Preprint
%arXiv:1203.6224.
\bibitem[LLT]{LLT} A.~Lascoux, B.~Leclerc, J.-Y.~Thibon; Hecke algebras at roots of unity and crystal bases of 
quantum affine algebras. Comm. Math. Phys. {\bf 181} (1996), no. 1, 205--263. 
%\bibitem[La]{La} A.~Lauda; A categorification of quantum ${\mathfrak sl}(2)$. Adv. Math. {\bf 225} (2010), no. 6, 
%3327--3424.
\bibitem[Le]{Le}  T.~Leinster; Basic bicategories. Preprint arXiv:math/9810017.
%\bibitem[Lu]{Lu} G.~Lusztig; Cells in affine Weyl groups.
%Algebraic groups and related topics (Kyoto/Nagoya, 1983),
%255--287, Adv. Stud. Pure Math., {\bf 6}, North-Holland, 
%Amsterdam, 1985. 
%\bibitem[McL]{McL} S.~Mac Lane; Categories for the working mathematician. 
%Second edition. Graduate Texts in Mathematics, {\bf 5}. Springer-Verlag, New York, 1998.
%\bibitem[Mat]{Mathas} A.~Mathas; Iwahori-Hecke algebras and Schur algebras of the 
%symmetric group. University Lecture Series, {\bf 15}. American Mathematical Society, Providence, RI, 1999.
%\bibitem[Ma]{Ma} V.~Mazorchuk; Lectures on algebraic categorification,
%Preprint arXiv:1011.0144, to appear in {\AA}rhus lecture notes series.
\bibitem[MM1]{MM} V.~Mazorchuk, V.~Miemietz; Cell $2$-representations of finitary
$2$-categories; Compositio Math. {\bf 147} (2011), 1519--1545.
\bibitem[MM2]{MM2} V.~Mazorchuk, V.~Miemietz; Additive versus abelian $2$-representations of 
fiat $2$-ca\-te\-go\-ri\-es. Preprint arxiv:1112.4949.
\bibitem[MM3]{MM3} V.~Mazorchuk, V.~Miemietz; Endomorphisms of cell $2$-representations. 
Preprint arXiv:1207.6236.
\bibitem[MOS]{MOS} V.~Mazorchuk, S.~Ovsienko, C.~Stroppel;
Quadratic duals, Koszul dual functors, and applications. 
Trans. Amer. Math. Soc. {\bf 361} (2009), no. 3, 1129--1172.
%\bibitem[MS1]{MS0} V.~Mazorchuk, C.~Stroppel; Translation and shuffling of projectively presentable modules and 
%a categorification of a parabolic Hecke module. Trans. Amer. Math. Soc. {\bf 357} (2005), no. 7, 2939--2973.
%\bibitem[MS1]{MS} V.~Mazorchuk, C.~Stroppel; Projective-injective 
%modules, Serre functors and symmetric algebras. J. Reine Angew. 
%Math. {\bf 616} (2008), 131--165. 
%\bibitem[MS2]{MS2} V.~Mazorchuk, C.~Stroppel; Categorification of 
%(induced) cell modules and the rough structure of generalised 
%Verma modules. Adv. Math. {\bf 219} (2008), no. 4, 1363--1426. 
%\bibitem[MS2]{MS3} V.~Mazorchuk, C.~Stroppel; A combinatorial approach 
%to functorial quantum $\mathfrak{sl}_k$ knot invariants. Amer. J. Math. 
%{\bf 131} (2009), no. 6, 1679--1713.
\bibitem[Mo]{Mo} K.~Morita; Duality for modules and its applications to the theory of rings 
with minimum condition. Sci. Rep. Tokyo Kyoiku Daigaku Sect. A {\bf 6} (1958) 83--142.
%\bibitem[Ne]{Ne} M.~Neunh{\"o}ffer; Kazhdan-Lusztig basis,
%Wedderburn decomposition, and Lusztig's homomorphism for
%Iwahori-Hecke algebras. J. Algebra {\bf 303} (2006), 
%no. 1, 430--446.
%\bibitem[Os]{Os} V.~Ostrik; Tensor ideals in the category of 
%tilting modules.  Transform. Groups  {\bf 2}  (1997),  no. 3, 279--287. 
\bibitem[Ri]{Ri} J.~Rickard; Morita theory for derived categories. J. London Math. Soc. 
(2) {\bf 39} (1989), no. 3, 436--456.
%\bibitem[Ro1]{Ro0} R.~Rouquier; Categorification of the 
%braid groups, Preprint arXiv:math/0409593. 
%\bibitem[Ro2]{Ro} R.~Rouquier; 2-Kac-Moody algebras. Preprint arXiv:0812.5023. 
%\bibitem[Sa]{Sa} B.~Sagan; The symmetric group. Representations, 
%combinatorial algorithms, and symmetric functions. Second edition. 
%Graduate Texts in Mathematics, {\bf 203}. Springer-Verlag, New York, 2001.
%\bibitem[So]{So} W.~Soergel; The combinatorics of Harish-Chandra bimodules. 
%J. Reine Angew. Math. {\bf 429} (1992), 49--74. 
\bibitem[St]{St} C.~Stroppel; Categorification of the Temperley-Lieb 
category, tangles, and cobordisms via projective functors. Duke Math. 
J. {\bf 126} (2005), no. 3, 547--596. 
\bibitem[To]{To} B.~To{\"e}n; The homotopy theory of dg-categories and derived Morita theory. 
Invent. Math. {\bf 167} (2007), no. 3, 615--667.
%\bibitem[Vi]{Vi} O.~Viro; Hyperfields for Tropical Geometry I. Hyperfields and 
%dequantization. Preprint arXiv:1006.3034.
\bibitem[Wi]{Wi} G.~Williamson; Singular Soergel bimodules. Int. Math. Res. Not. 2011, no. {\bf 20}, 4555--4632.
\end{thebibliography}
\end{document}